\documentclass[12pt,reqno]{article}
\usepackage{diagbox}
\usepackage{amsfonts}
\usepackage{booktabs}
\usepackage{bbm}
\usepackage{bm}
\usepackage{pifont}
\usepackage{mathptmx}
\usepackage{amsmath}
\usepackage{mathrsfs}
\usepackage{hyperref}
\usepackage{amssymb}
\usepackage{amsmath,color}
\usepackage{amsthm}
\usepackage{amstext}
\usepackage[numbers,sort&compress]{natbib}
\usepackage{amsopn}
\usepackage{texdraw}
\usepackage{graphicx}
\usepackage{multirow}
\usepackage{lscape}
\usepackage{float}
\usepackage[palatino,gill,courier]{altfont}
\usepackage[left=3cm,right=3cm,top=3cm,bottom=3cm]{geometry}
\usepackage{esint}
\usepackage{setspace}

\usepackage[english]{babel}
\usepackage{mathrsfs,amsmath}
\usepackage{color}
\usepackage{framed}
\allowdisplaybreaks

\usepackage{color} 
\usepackage{enumerate} 
\newtheorem{theorem}{Theorem}[section] 

\newtheorem{remark}[theorem]{Remark}

\let\Section=\section
\def\section{\setcounter{equation}{0}\Section}
\title{Asymmetric fractional coupled magnetizable piezoelectric beams with infinite viscoelastic memory: polynomial decay and sharpness of the decay rate}
\author{Jun Zhou\footnote{Corresponding author. Email: jzhou@swu.edu.cn} ~~~~~~ Min Wang\footnote{Email: 1909085317@qq.com}\\\\{School of Mathematics and Statistics, Southwest University},\\
	{Chongqing 400715, People's Republic of China}}
\date{}
\begin{document}
	\maketitle
	\begin{abstract}
		This paper studies the long-time dynamics of a coupled hyperbolic system for magnetizable piezoelectric beams with infinite viscoelastic memory. Memory dissipation is characterized by $A^\alpha$, the fractional power of a positive self-adjoint operator $A$ with $\alpha\in[0,1)$. An asymmetric fractional magnetoelectric coupling is adopted: the mechanical-to-magnetic coupling uses integer-order operator $A$, while the magnetic feedback to mechanics is governed by fractional operator $A^\beta$ ($\beta\in[0,1)$). This model bridges the gap between the well-studied integer coupling case ($\beta=1$) and the unsolved fully fractional symmetric coupling problem, offering a universal framework for related coupled systems.

Under mild assumptions on memory kernels and system parameters, we prove well-posedness via semigroup theory and derive an explicit polynomial decay estimate for smooth initial data:
\[
\|X(t)\|_{\mathcal H} \le C t^{-\frac{1}{4-2\beta-2\alpha}}\|X_0\|_{D(\mathcal A)},\quad \forall\,t\ge 1,
\]
where the decay exponent is explicitly determined by $\alpha$ and $\beta$. For exponentially decaying memory kernels, the decay rate is sharp if stiffness coefficients satisfy $\alpha_1\ne\alpha_2$. When $\alpha_1=\alpha_2$, we only obtain an upper bound $\delta\le \frac{1}{3-\beta-2\alpha}$ for decay index $\delta$, leaving the optimal rate open.

Comparisons with integer feedback coupling ($\beta=1$) show fractional feedback ($\beta<1$) slows energy decay. It demonstrates that $\beta$ weakens indirect damping and degrades structural stabilization. The results reveal the intrinsic interaction between fractional memory dissipation and fractional coupling in strongly coupled dissipative systems.
	\end{abstract}
	\section{Introduction}\label{sec1}
The piezoelectric effect refers to a physical phenomenon that realizes the coupling between mechanical stress and electrical signals. Based on this fundamental property, piezoelectric beams enable the bidirectional conversion between mechanical and electrical energy. As a typical category of intelligent functional structures, piezoelectric beams have been extensively applied to mechanical engineering, biomedical devices and industrial manufacturing, attracting considerable research attention across materials science and related engineering disciplines \cite{Dagdeviren2016,s18113973,9164928}.
Models of magnetizable piezoelectric beams serve as canonical dynamic systems to describe the multi-physical coupling behaviors among mechanical, electric and magnetic fields. Along with the rapid advancement of magnetoelectric coupling research, numerous efforts have been devoted to the mathematical modeling and stability analysis of magnetizable piezoelectric beam systems. In pioneering studies, Morris and \"Ozer \cite{Morris2013,Morris2014,Ozer2015} developed a comprehensive fully dynamic governing system for such beams, which is formulated as follows:
\begin{equation}\label{morris}
			\begin{cases}
				\rho v_{tt}(x,t) = \alpha v_{xx}(x,t) - \gamma\beta p_{xx}(x,t), & x\in(0,L),\ t>0, \\[4pt]
				\mu p_{tt}(x,t) = \beta p_{xx}(x,t) - \gamma\beta v_{xx}(x,t), & x\in(0,L),\ t>0, \\[4pt]
				v(0,t) = p(0,t) = 0, & t>0, \\[4pt]
				\alpha v_x(L,t) - \gamma\beta p_x(L,t) = 0, & t>0, \\[4pt]
				\beta p_x(L,t) - \gamma\beta v_x(L,t) = -\dfrac{V(t)}{h}, & t>0.
			\end{cases}
		\end{equation}
The physical interpretations of all variables and parameters involved in Eq. \eqref{morris} are specified in Table \ref{table1}.
\begin{table}[htbp]\label{table1}
  \centering
  \caption{Nomenclature for the magnetizable piezoelectric beam model}
  \label{tab:nomenclature}
  \begin{tabular}{ll}
    \toprule
    Symbol & Physical Interpretation \\
    \midrule
    $x$, $t$ & Spatial coordinate along the beam length and time variable, respectively \\
    $L$ & Total length of the magnetizable piezoelectric beam \\
    $v(x,t)$ & Transverse mechanical displacement of the beam at position $x$ and time $t$ \\
    $p(x,t)$ & Magnetic potential function induced by the magnetoelectric coupling effect \\
    $\rho$ & Linear mass density of the piezoelectric beam material \\
    $\mu$ & Magnetic permeability coefficient associated with magnetic field variations \\
    $\alpha$ & Structural stiffness coefficient of the beam, reflecting mechanical elastic properties \\
    $\beta$ & Dielectric constant that characterizes the electrical response of piezoelectric materials \\
    $\gamma$ & Magnetoelectric coupling coefficient bridging mechanical, electric and magnetic fields \\
    $h$ & Equivalent thickness of the piezoelectric beam \\
    $V(t)$ & External applied voltage loaded at the free end of the beam \\
    \bottomrule
  \end{tabular}
\end{table}
In their research, the authors established the coupled dynamic model via the variational principle. Furthermore, they enhanced the structural stability of the system by adopting electric feedback control strategies, and demonstrated that current-driven actuation possesses prominent merits regarding the boundedness of control operators. Subsequent scholars have further enriched relevant theories from diverse perspectives. Ramos et al. \cite{MR3808160} studied one-dimensional dissipative piezoelectric beams with magnetic effects and verified that damping ensures exponential stability. Messaoudi et al. \cite{MR4853207} derived parameter-independent universal energy decay estimates for fully dynamic electrostatic piezoelectric beams incorporating magnetic effects and viscoelastic damping. Douib et al. \cite{MR4831270} explored the dynamic characteristics of magnetizable piezoelectric beams with time-delay perturbations and proved the exponential stabilization of the considered system. Loucif et al. \cite{MR4970606} integrated magnetizable piezoelectric beam equations with the Green-Naghdi type III thermoelastic model, and established a complete exponential stability framework for time-delay coupled multi-physical systems. Al-Gharabli \cite{MR4901041} studied piezoelectric beams with infinite memory and nonlinear frictional feedback, establishing general decay rates without parameter constraints.
	
Now we introduce some results on some abstract magnetizable piezoelectric systems. Let $H$ denote a complex Hilbert space equipped with inner product $\langle\cdot,\cdot\rangle$, which induces the associated norm $\|\cdot\|$. Suppose $A:D(A)\to H$ is a strictly positive, self-adjoint linear operator on $H$ admitting a compact resolvent, with $D(A)$ denoting the domain of $A$. The spectrum of $A$ consists solely of isolated eigenvalues $\{\gamma_n\}_{n\in\mathbb{N}}$ obeying
\[
0 < \gamma_1 < \gamma_2 < \cdots < \gamma_n < \gamma_{n+1} < \cdots,\qquad \lim_{n\to\infty}\gamma_n = \infty.
\]
For each $n\in\mathbb{N}$, let $e_n\in D(A)$ stand for the normalized eigenvector associated to $\gamma_n$, satisfying $\|e_n\|=1$ and $A e_n = \gamma_n e_n$. The sequence $\{e_n\}_{n=1}^\infty$ constitutes an orthonormal basis of $H$.

Set $D(A^0)\triangleq H$. For any $\alpha\in(0,1]$, we define
\begin{equation}\label{DAalpha}
D(A^{\alpha/2})\triangleq\left\{\phi\in H\,\bigg|\,\|A^{\alpha/2}\phi\|<\infty\right\},
\end{equation}
where the fractional power $A^{\alpha/2}$ acts on arbitrary $\phi\in H$ via the series expansion
\[
A^{\alpha/2}\phi=\sum_{n=1}^\infty\gamma_n^{\alpha/2}\langle\phi,e_n\rangle e_n,
\]
with the corresponding norm identity
\[
\|A^{\alpha/2}\phi\|^2=\sum_{n=1}^\infty\gamma_n^\alpha\bigl|\langle\phi,e_n\rangle\bigr|^2.
\]
The space $D(A^{\alpha/2})$ is readily seen to be a Hilbert space under the inner product
\begin{equation}\label{innerDAalpha}
\langle\phi_1,\phi_2\rangle_{D(A^{\alpha/2})}\triangleq\left\langle A^{\alpha/2}\phi_1,A^{\alpha/2}\phi_2\right\rangle,\qquad
\forall\,\phi_1,\phi_2\in D(A^{\alpha/2}).
\end{equation}

Zhang et al. \cite{MR4450079} studied a multi-dimensional nonlinear piezoelectric beam with viscoelastic infinite memory defined on a bounded smooth domain $\Omega \subset \mathbb{R}^n$ with piecewise smooth boundary $\partial\Omega = \Gamma_0 \cup \Gamma_1$, where $\Gamma_0$ and $\Gamma_1$ are disjoint boundary components. The governing coupled hyperbolic partial differential equations read
\begin{equation}\label{ZAMP-original}
\begin{cases}
\rho v_{tt}(x,t) = \alpha \Delta v(x,t) - \gamma\beta \Delta p(x,t) + f_1(v,p) - \displaystyle\int_{0}^{\infty} g(s)\Delta v(x,t-s)\,\mathrm{d}s, & x\in\Omega,\ t>0,\\
\mu p_{tt}(x,t) = \beta \Delta p(x,t) - \gamma\beta \Delta v(x,t) + f_2(v,p), & x\in\Omega,\ t>0,\\
v(x,t) = p(x,t) = 0, & x\in\Gamma_0,\ t>0,\\
\alpha \dfrac{\partial v}{\partial \mathbf{n}}(x,t) - \gamma\beta \dfrac{\partial p}{\partial \mathbf{n}}(x,t) = \beta \dfrac{\partial p}{\partial \mathbf{n}}(x,t) - \gamma\beta \dfrac{\partial v}{\partial \mathbf{n}}(x,t) = 0, & x\in\Gamma_1,\ t>0,\\
v(x,0) = v_0(x),\ v_t(x,0) = v_1(x),\ p(x,0) = p_0(x),\ p_t(x,0) = p_1(x), & x\in\Omega,\\
v(x,-s) = h(x,s), & x\in\Omega,\ s>0,
\end{cases}
\end{equation}
where $v(x,t)$ is transverse displacement, $p(x,t)$ magnetic potential, $h(s)$ memory prehistory, and $f_1,f_2$ locally Lipschitz nonlinear source terms, $g(s)$ integrable viscoelastic kernel satisfying the following assumptions:
 \begin{description}
		\item[(A1)]
	The memory kernel $g\in L^1(\mathbb{R}_+)\cap H^1(\mathbb{R}_+)$ satisfy $0<\zeta:=\int_0^\infty g(s)\mathrm{d}s<\infty$ and $g(s)>0$ for all $s\in\mathbb{R}_+$. In addition, $g'(s)<0$ holds almost everywhere on $\mathbb{R}_+$, and there exist positive constants $k_0,k_1$ such that
		\[
		-k_0 g(s)\leq g'(s)\leq -k_1 g(s).
		\]
		\end{description}
This model uses integer Laplace operators with infinite viscoelastic memory solely acting on the mechanical equation; the magnetic equation achieves indirect stabilization through magnetoelectric coupling. The authors converted the PDE system to an abstract Hilbert-space evolution equation. Using the L\"umer-Phillips theorem, they verified the linear operator generates a contractive $C_0$-semigroup, and applied the Banach fixed-point theorem to obtain global well-posedness for small initial data. Frequency-domain resolvent estimates together with energy multipliers yielded uniform exponential energy decay. At the end of the paper, the authors put forward several open research problems, one of them is to investigate the large-time stability and quantitative energy decay rates for the following abstract fractional coupled magnetizable piezoelectric system with fractional memory and fractional magnetoelectric coupling:
\begin{equation}\label{open-model}
\begin{cases}
v_{tt}(t) + a_1 A v(t) - \kappa A^\beta p(t) - \displaystyle\int_{0}^{\infty} g(s) A^\alpha v(t-s)\,\mathrm{d}s = 0,\\
p_{tt}(t) + a_2 A p(t) - \kappa A^\beta v(t) = 0,\quad \alpha,\beta\in[0,1).
\end{cases}
\end{equation} This abstract system generalizes the linearized version of the original PDE model \eqref{ZAMP-original}: if we set $\alpha=\beta=1$ and take $A=-\Delta$, then \eqref{open-model} exactly recovers the linear part of the concrete multi-dimensional piezoelectric beam model studied in \cite{MR4450079}. The fractional exponent $\alpha$ governs the viscoelastic memory dissipation behavior of mechanical displacement, while $\beta$ characterizes the strength of fractional-order magnetoelectric cross-coupling between the mechanical and magnetic components. The memory integral $\int_0^\infty g(s) A^\alpha v(t-s)\mathrm{d}s$ is called \lq\lq viscoelastic infinite memory\rq\rq, and many researchers have investigated problems concerning this type of memory term. For example, Wang et al. \cite{MR3109859} investigated a wave equation with viscoelastic infinite memory and proved the exponential stability of the system via spectral analysis and the Riesz basis method. Guesmia et al. \cite{MR4125294} established improved general decay rates for two types of viscoelastic wave equations with infinite memory under more general assumptions on the kernel functions.
The authors of \cite{MR4450079} pointed out that both fractional parameters $\alpha$ and $\beta$ jointly modulate the overall dissipative mechanism of the coupled hyperbolic system, yet the precise quantitative relationship between $(\alpha,\beta)$ and the sharp energy decay rate of system solutions remains completely unresolved. The core target of this open problem is to rigorously characterize how the two independent fractional orders $\alpha,\beta\in[0,1)$ jointly determine the optimal decay speed of solutions to the abstract coupled evolution system \eqref{open-model}.

To partially address this unresolved open problem, Zhang et al. \cite{MR4939108} carried out a follow-up study. Their work considered the special case where the magnetoelectric fractional order is fixed as $\beta=1$, while retaining fractional viscoelastic memory $A^\alpha$ with $\alpha\in[0,1)$. The abstract linear system analyzed in their paper reads
\begin{equation}\label{ESAIM-model}
\begin{cases}
\rho v_{tt}(t)=-\alpha Av(t)+\gamma\beta Ap(t)+\displaystyle\int_{0}^{\infty}g(s)A^\alpha v(t-s)\,\mathrm{d}s, & t>0,\\
\mu p_{tt}(t)=-\beta Ap(t)+\gamma\beta Av(t), & t>0,\\
v(0)=v_0,\ v_t(0)=v_1,\ p(0)=p_0,\ p_t(0)=p_1,\\
v(-s)=h(s), & s>0,
\end{cases}
\end{equation}
where the memory kernel $g(s)$ satisfies identical integrability and monotonicity assumptions (A1) and $\alpha>\gamma_1^2\beta$. Following the approach developed in Dafermos \cite{Dafermos1970} and Fabrizio et al. \cite{MR2679371}, the authors reformulated the system into an extended state space incorporating the memory history variable $\eta(t,s)=v(t)-v(t-s)$, then proved the associated operator generates a contractive $C_0$-semigroup via L\"umer-Phillips arguments, establishing full well-posedness on the weighted Hilbert space. Using refined frequency-domain resolvent bounds along spectral asymptotic analysis, they rigorously derived an explicit, polynomial optimal decay rate $t^{-1/(2-2\alpha)}$ solely controlled by the memory fractional index $\alpha$, and further validated the sharpness of this decay estimate through detailed eigenvalue asymptotic expansions under exponential-type memory kernels. Furthermore, for more models with fractional memory terms, Oliveira et al. \cite{MR4663194,MR4991955} studied a wave-plate coupled system subject to fractional memory dissipation, established that the polynomial decay rates depend explicitly on the fractional order parameters, and proved the optimality of the obtained decay rates. Nevertheless, this follow-up contribution of \cite{MR4939108} only resolves the simplified scenario with integer-order magnetoelectric coupling $\beta=1$, leaving the general case with sub-unit fractional coupling $\beta\in[0,1)$ entirely unaddressed, which motivates our current investigation.

To fill this theoretical gap and provide a seamless transition from the fully integer-coupled case in \cite{MR4939108} toward the fully fractional symmetric system proposed in the open problem \eqref{open-model}, we investigate an intermediate asymmetric fractional coupled magnetizable piezoelectric system in the present work. Our model adopts a hierarchical fractional coupling setting: the coupling term acting on the mechanical displacement equation adopts integer-order operator $A$ (i.e., $\beta=1$), while the feedback coupling term on the magnetic potential equation retains general fractional operator $A^\beta$ with $\beta\in[0,1)$. The complete abstract evolution system with infinite viscoelastic memory and initial/pre-history data is given by
\begin{equation}\label{model}
		\begin{cases}
			\displaystyle v_{tt}(t) = -\alpha_1 A v(t) + k A p(t) + \int_0^\infty g(s) A^\alpha v(t-s) \,\mathrm{d}s, & t>0, \\
			p_{tt}(t) = -\alpha_2 A p(t) + k A^\beta v(t), & t>0 ,\\
			v(-s) = h(s), & s>0, \\
			v(0)=v_0:=\lim_{s\to0^+}h(s),\ v_t(0)=v_1,\ p(0)=p_0,\ p_t(0)=p_1,
		\end{cases}
	\end{equation}
where $\alpha\in[0,1)$ stands for the fractional order of viscoelastic memory dissipation, and $\beta\in[0,1)$ describes the fractional magnetoelectric feedback coupling acting on the magnetic potential component, we make the following assumptions on the parameters:
 \begin{description}
\item[(A2)]
		Assume that constants $\alpha\in[0,1)$, $\beta\in[0,1)$, $\alpha_1>\zeta\gamma_1^{\alpha-1}$, $\alpha_2>0$ and $k>0$ satisfy
		\[
		\alpha_2\big(\alpha_1 - \zeta \gamma_1^{\alpha-1}\big)
		> k^2\big(1 + \gamma_1^{\beta-1}\big)^2.
		\]
 \end{description}

 Compared with the symmetric fully fractional system \eqref{open-model} proposed as an open problem in \cite{MR4450079}, our asymmetric structure constitutes a natural transitional model: it retains the simplified integer coupling from mechanical to magnetic field analyzed in \cite{MR4939108}, while introducing the sub-unit fractional coupling $\beta\in[0,1)$ from magnetic potential back to mechanical displacement, which allows us to separately quantify how the fractional feedback term $A^\beta$ modulates the overall energy decay rate without the interference of two-way fractional cross-coupling. This asymmetric setup bridges the existing literature and the unsolved general open problem, enabling us to rigorously characterize the individual contribution of the fractional coupling order $\beta$ to the stabilization performance of magnetizable piezoelectric systems with fractional viscoelastic memory.

\noindent\textbf{Summary of our main results.} In this work, we systematically investigate the asymmetric fractional coupled system \eqref{model}. Our main contributions are threefold. First, under assumptions (A1) and (A2), we establish the well-posedness of the system via semigroup theory (Theorem \ref{thmwell}), proving that the associated operator generates a contractive \(C_0\)-semigroup on the extended Hilbert space \(\mathcal{H}\). Second, by means of refined frequency-domain resolvent estimates and multiplier techniques, we derive an explicit polynomial decay rate for smooth initial data:
\[
\|X(t)\|_{\mathcal{H}} \le C t^{-\frac{1}{4-2\beta-2\alpha}} \|X_0\|_{D(\mathcal{A})}, \qquad t\ge 1,
\]
where the decay exponent explicitly depends on both the memory fractional order \(\alpha\in[0,1)\) and the magnetoelectric coupling fractional order \(\beta\in[0,1)\) (Theorem \ref{thmdecay}). Third, for exponentially decaying memory kernels, we prove that the above decay rate is sharp whenever the stiffness coefficients satisfy \(\alpha_1\ne \alpha_2\) (Theorem \ref{thmoptimaldecay}); however, when \(\alpha_1=\alpha_2\), our analysis only yields an upper bound \(\delta \le 1/(3-\beta-2\alpha)\) for the possible decay index $\delta$, leaving the exact optimal rate as an open problem (Remark \ref{rem-sharpness}).

\medskip
\noindent\textbf{Comparison with the prior work \cite{MR4939108}.}
A direct comparison with the preceding study highlights the crucial role played by the fractional coupling order \(\beta\). In \cite{MR4939108}, the magnetoelectric coupling was restricted to the integer order \(\beta=1\), and the obtained optimal decay rate was \(t^{-1/(2-2\alpha)}\), which corresponds exactly to setting \(\beta=1\) in our general exponent \(1/(4-2\beta-2\alpha)\). For any \(\beta\in[0,1)\), our new exponent satisfies
\[
\frac{1}{4-2\beta-2\alpha} < \frac{1}{2-2\alpha},
\]
since the denominator increases with \(\beta\). Thus, the introduction of the fractional feedback coupling (with \(\beta<1\)) slows down the decay rate compared to the integer-coupling case. This confirms that the fractional order \(\beta\) directly weakens the indirect damping effect transmitted from the mechanical to the magnetic component, thereby deteriorating the overall stabilization performance. Moreover, our sharpness result reveals that when \(\alpha_1\ne\alpha_2\), the decay exponent is indeed optimal, whereas the symmetric case \(\alpha_1=\alpha_2\) exhibits a loss of sharpness, a phenomenon not observed in \cite{MR4939108} because there \(\alpha_1\) and \(\alpha_2\) were effectively distinct due to the absence of symmetry. Therefore, our work not only generalizes the decay rate formula to incorporate \(\beta\), but also uncovers a new subtlety regarding the influence of stiffness coefficients on optimality, which was not addressed in the previous study.
		
The structure of this paper is as follows. In Section \ref{sec2}, we construct the state space and prove the well-posedness of the system. In Section \ref{sec3}, we establish polynomial stability using frequency-domain estimates. In Section \ref{sec4}, we verify the optimality of the decay rate through spectral analysis.
	\section{Well-posedness}\label{sec2}
	
	This section focuses on the well-posedness of the system \eqref{model} by $C_0$-semigroup theory.
	
	First, we introduce the state space $\mathcal{H}$ as follows:
	\[
	\mathcal{H} \triangleq D(A^{1/2}) \times H \times D(A^{1/2}) \times H \times L_g^2\big(\mathbb{R}_+, D(A^{{\alpha}/{2}})\big),
	\]
	where the space $D(A^{{\alpha}/{2}})$ is defined in \eqref{DAalpha} and
	\begin{align*}
		L_g^2\big(\mathbb{R}_+, D(A^{{\alpha}/{2}})\big)\triangleq\left\{\psi:\mathbb{R}_+\mapsto D(A^{{\alpha}/{2}})\bigg|\|\psi\|_{L_g^2\big(\mathbb{R}_+, D(A^{{\alpha}/{2}})\big)}^2=\int_0^\infty g(s)\|A^{\frac\alpha2}\psi(s)\|^2ds\right\}.
	\end{align*}
	It is obvious that $L_g^2\big(\mathbb{R}_+, D(A^{{\alpha}/{2}})\big)$ ($L_g^2$ for short) is also a Hilbert space with inner product
	\[\langle \psi_1,\psi_2\rangle_{L^2_g}\triangleq\int_0^\infty g(s)\langle A^{\frac\alpha 2}\psi_1(s), A^{\frac\alpha 2}\psi_2(s)\rangle,~~~\forall \psi_1,\psi_2\in L^2_g.\]
	The inner product on $\mathcal{H}$ is defined by
	\begin{align}\label{innerH}
		\langle (v,u,p,q,\eta)^\top , (\widetilde{v},\widetilde{u},\widetilde{p},\widetilde{q},\widetilde{\eta})^\top  \rangle_{\mathcal{H}}
		\triangleq& \alpha_1 \langle A^{1/2}v, A^{1/2}\widetilde{v} \rangle - \zeta \langle A^{\frac{\alpha}{2}}v, A^{\frac{\alpha}{2}}\widetilde{v} \rangle + \langle u,\widetilde{u} \rangle \notag\\
		&+ \alpha_2 \langle A^{1/2}p, A^{1/2}\widetilde{p} \rangle -k\operatorname{Re}\langle A^{\frac{1}{2}} v, A^{\frac{1}{2}}\widetilde p\rangle
		-k\operatorname{Re}\langle A^{\frac{\beta}{2}} p, A^{\frac{\beta}{2}} \widetilde v\rangle\notag\\
		&+ \langle q,\widetilde{q} \rangle + \int_0^\infty g(s) \langle A^{\frac{\alpha}{2}}\eta(s), A^{\frac{\alpha}{2}}\widetilde{\eta}(s) \rangle \, ds.
	\end{align}
	Thus, the norm introduced by the inner product is given by
	\begin{align}\label{normH}
		\|(v,u,p,q,\eta)\|_{\mathcal{H}}^2
		\triangleq& \alpha_1 \|A^{1/2}v\|^2 - \zeta\|A^{\frac{\alpha}{2}}v\|^2 + \|u\|^2+ \alpha_2 \|A^{1/2}p\|^2 \notag\\
		&- k \operatorname{Re} \langle A^{\frac{1}{2}} v, A^{\frac{1}{2}} p \rangle
		- k \operatorname{Re} \langle A^{\frac{\beta}{2}} v, A^{\frac{\beta}{2}} p \rangle+ \|q\|^2 + \|\eta\|_{L_g^2}^2.
	\end{align}
	
	Next we show the norm defined in \eqref{normH} is equivalent to the standard norm of $\mathcal{H}$.
	\begin{theorem}\label{thmequinorm}
		There exists two positive constants $\kappa_1$ and $\kappa_2$ independent of $X=(v,u,p,q,\eta)\in\mathcal{H}$ such that
		$$\kappa_1\|X\|_{\rm stan}\le \|X\|_{\mathcal{H}}\le\kappa_2\|X\|_{\rm stan},$$
		where
		$$\|X\|_{\rm stan}^2\triangleq\|A^{1/2}v\|^2+\|u\|^2+\|A^{1/2}p\|^2+\|q\|^2+\|\eta\|_{L^2_g}^2.$$
	\end{theorem}
	\begin{proof}
		Since by the assumption (A2), the constants $\alpha_1>\zeta\gamma_1^{\alpha-1}$, $\alpha_2>0$ and $k>0$ satisfy
		\(
		\alpha_2\big(\alpha_1 - \zeta \gamma_1^{\alpha-1}\big)
		> k^2\big(1 + \gamma_1^{\beta-1}\big)^2,
		\)
		we can choose a positive constant $\varepsilon$ such that
		\[\dfrac{2k\left(1 + \gamma_1^{\beta-1}\right)}{\alpha_2}<\varepsilon< \dfrac{2\left(\alpha_1 - \zeta \gamma_1^{\alpha-1}\right)}{k\left(1 + \gamma_1^{\beta-1}\right)}.\]
		Then
		\[\delta_1\triangleq\alpha_1 - \zeta \gamma_1^{\alpha-1} - \dfrac{k\varepsilon}{2}
		\left(1 + \gamma_1^{\beta-1}\right),~~~\delta_2\triangleq \alpha_2 - \dfrac{2k}{\varepsilon}
		\left(1 + \gamma_1^{\beta-1}\right) > 0.\]
		Moreover, since $\alpha\le 1$, we get
		\begin{align*}
			\|A^{\frac{\alpha}{2}}v\|^2=\sum_{n=1}^\infty\gamma_n^\alpha|\langle v,e_n\rangle|^2=\sum_{n=1}^\infty\gamma_n^{\alpha-1}\gamma_n|\langle v,e_n\rangle|^2\le\gamma_1^{\alpha-1}\sum_{n=1}^\infty\gamma_n|\langle v,e_n\rangle|^2=\gamma_1^{\alpha-1}\|A^{\frac{1}{2}}v\|^2.
		\end{align*}
		Applying Young's inequality yields
		\begin{align}\label{coer}
			& \alpha_1 \|A^{\frac{1}{2}}v\|^2 - \zeta \|A^{\frac{\alpha}{2}}v\|^2 + \alpha_2 \|A^{\frac{1}{2}}p\|^2
			- k\operatorname{Re}\langle A^{\frac{1}{2}}v, A^{\frac{1}{2}}p\rangle
			- k\operatorname{Re}\langle A^{\frac{\beta}{2}}v, A^{\frac{\beta}{2}}p\rangle \notag\\
			& \ge \alpha_1 \|A^{\frac{1}{2}}v\|^2 - \zeta\gamma_1^{\alpha-1} \|A^{\frac{1}{2}}v\|^2 + \alpha_2 \|A^{\frac{1}{2}}p\|^2  - k\left(\frac{\varepsilon}{2}\|A^{\frac{1}{2}}v\|^2 + \frac{2}{\varepsilon}\|A^{\frac{1}{2}}p\|^2\right)
			- k\left(\frac{\varepsilon}{2}\|A^{\frac{\beta}{2}}v\|^2 + \frac{2}{\varepsilon}\|A^{\frac{\beta}{2}}p\|^2\right)  \notag\\
			& \ge \alpha_1 \|A^{\frac{1}{2}}v\|^2 - \zeta\gamma_1^{\alpha-1} \|A^{\frac{1}{2}}v\|^2 + \alpha_2 \|A^{\frac{1}{2}}p\|^2  - k\frac{\varepsilon}{2}\|A^{\frac{1}{2}}v\|^2 - \frac{2k}{\varepsilon}\|A^{\frac{1}{2}}p\|^2
			- \frac{k\varepsilon}{2}\gamma_1^{\beta-1}\|A^{\frac{1}{2}}v\|^2
			- \frac{2k}{\varepsilon}\gamma_1^{\beta-1}\|A^{\frac{1}{2}}p\|^2  \notag\\
			& = \delta_1\|A^{\frac{1}{2}}v\|^2+ \delta_2\|A^{\frac{1}{2}}p\|^2.
		\end{align}
		Therefore, we get
		\begin{align}\label{dj1}
			\|X\|_{\mathcal{H}}^2=&\alpha_1 \|A^{1/2}v\|^2 - \zeta\|A^{\frac{\alpha}{2}}v\|^2 + \|u\|^2+ \alpha_2 \|A^{1/2}p\|^2 - k \operatorname{Re} \langle A^{\frac{1}{2}} v, A^{\frac{1}{2}} p \rangle
			- k \operatorname{Re} \langle A^{\frac{\beta}{2}} v, A^{\frac{\beta}{2}} p \rangle+ \|q\|^2 + \|\eta\|_{L_g^2}^2\notag\\
			\ge&\delta_1\|A^{\frac{1}{2}}v\|^2+ \delta_2\|A^{\frac{1}{2}}p\|^2+\|u^2\|+\|q\|^2 + \|\eta\|_{L_g^2}^2\notag\\
			\ge&\min\{\delta_1,\delta_2,1\}\|X\|_{\rm stan}^2.
		\end{align}
		
		On the other hand, by Young's inequality, it is obvious that
		\begin{align}\label{dj2}
			\|X\|_{\mathcal{H}}^2=&\alpha_1 \|A^{1/2}v\|^2 - \zeta\|A^{\frac{\alpha}{2}}v\|^2 + \|u\|^2+ \alpha_2 \|A^{1/2}p\|^2 - k \operatorname{Re} \langle A^{\frac{1}{2}} v, A^{\frac{1}{2}} p \rangle
			- k \operatorname{Re} \langle A^{\frac{\beta}{2}} v, A^{\frac{\beta}{2}} p \rangle+ \|q\|^2 + \|\eta\|_{L_g^2}^2\notag\\
			\le&\left(\alpha_1+\frac k2+\frac k2\gamma_1^{\beta-1}\right)\|A^{\frac{1}{2}}v\|^2+ \left(\alpha_2+\frac k2+\frac k2\gamma_1^{\beta-1}\right)\|A^{\frac{1}{2}}p\|^2+\|u^2\|+\|q\|^2 + \|\eta\|_{L_g^2}^2\notag\\
			\le&\max\left\{\alpha_1+\frac k2+\frac k2\gamma_1^{\beta-1},\alpha_2+\frac k2+\frac k2\gamma_1^{\beta-1},1\right\}\|X\|_{\rm stan}^2.
		\end{align}
		By \eqref{dj1} and \eqref{dj2}, we get the conclusion.
	\end{proof}
	
	Next, to transform the non-autonomous system \eqref{model} to an autonomous system we define a new variable $\eta(t,s)$ as
	\[
	\eta(t,s) = v(t) - v(t-s), \quad t,s>0
	\]
	Then, the first equation of the system \eqref{model} may be reformulated as follows:
	$$v_{tt}(t) = -\alpha_1 A v(t) + k A p(t)+ \zeta A^\alpha v(t) - \int_{0}^{\infty} g(s) A^\alpha \eta(s) \,\mathrm{d}s.$$
	Therefore, we can re-written the system \eqref{model} as the following abstract autonomous system in $\mathcal{H}$:
	\begin{equation}\label{modelmain}
		\begin{cases}
			\dfrac{dX(t)}{dt}=\mathcal{A} X(t), & t>0, \\
			X(0)=X_0\triangleq \big(v_0,v_1,p_0,p_1,\eta_0\big)^\top,
		\end{cases}
	\end{equation}
	where $X(t)\triangleq\big(v(t),u(t),p(t),q(t),\eta(t,s)\big)^\top$ with $u(t)\triangleq v_t(t)$ and $q(t)\triangleq p_t(t)$, $\eta_0\triangleq\eta(0,s)=v_0-h(s)$, and $\mathcal{A}:D(\mathcal{A})\subset \mathcal{H}\mapsto\mathcal{H}$ is a linear operator given by
	\begin{equation}\label{dfA}
		\mathcal{A}\begin{pmatrix} v \\ u \\ p \\ q \\ \eta \end{pmatrix}
		\triangleq
		\begin{pmatrix}
			u \\
			-\alpha_1 A v+k Ap+\zeta A^\alpha v-\int_0^\infty g(s) A^\alpha \eta(s) ds \\
			q \\
			-\alpha_2 A p+k A^\beta v \\
			u-\eta_s(s)
		\end{pmatrix},
	\end{equation}
	where $(v,u,p,q,\eta)^\top\in D(\mathcal{A})$ and
	\begin{align*}
D(\mathcal{A})\triangleq&\left\{X=(v,u,p,q,\eta)^\top\in\mathcal{H}\bigg|\mathcal{A}X\in H\right\}\\
		=&\Bigg\{X=(v,u,p,q,\eta)^\top \bigg| u,v,p,q\in D(A^{1/2}),~~\eta\in \mathcal{M},\\
		&-\alpha_1 A v+k A p+\zeta A^\alpha v-\int_0^\infty g(s) A^\alpha \eta(s) ds \in H,~~-\alpha_2 A p+k A^\beta v\in H \Bigg\}.
	\end{align*}
	Here
	\[
	\mathcal{M}\triangleq \left\{ \eta(s)\in L_g^2,\ \eta_s(s)\in L_g^2,\ \eta(0)=0 \right\}.
	\]
	The main result of this section is the following theorem:
	\begin{theorem}\label{thmwell}
		Suppose that assumptions (A1) and (A2) hold.
		The operator $\mathcal{A}$ generates a $C_0$ semigroup of contractions $\{S(t)\}_{t\geq0}$ on the Hilbert space $\mathcal{H}$.
		Then, for any $X_0\in\mathcal{H}$, problem \eqref{modelmain} admits a unique mild solution
		\[
		X(t)\triangleq S(t)X_0\in C([0,\infty);\mathcal{H}).
		\]
		Moreover, if $X_0\in D(\mathcal{A})$, then
		\[
		X(t)\in C([0,\infty); D(\mathcal{A})) \cap C^1([0,\infty); \mathcal{H})
		\]
		is the classical solution of problem \eqref{modelmain}.
	\end{theorem}
	
	\begin{proof}
		First,for any $X = (v, u, p, q, \eta)^\top \in {D}(\mathcal{A})$, a direct calculation yields that
		\begin{align}\label{dissipative}
			\operatorname{Re}\langle \mathcal{A}X, X \rangle
			&= \operatorname{Re}\Bigg\{ \alpha_1 \langle A^{\frac{1}{2}}u, A^{\frac{1}{2}}v \rangle - \zeta \langle A^{\frac{\alpha}{2}}u, A^{\frac{\alpha}{2}}v \rangle+ \langle -\alpha_1 A v + k A p + \zeta A^\alpha v - \int_0^\infty g(s)A^\alpha \eta(s)\mathrm{d}s, u \rangle \notag\\
			&\quad + \alpha_2 \langle A^{\frac{1}{2}}q, A^{\frac{1}{2}}p \rangle-k\langle A^{\frac{1}{2}} u, A^{\frac{1}{2}} p\rangle
			-k\langle A^{\frac{\beta}{2}} q, A^{\frac{\beta}{2}} v\rangle\notag\\
			&\quad + \langle -\alpha_2 A p + k A^\beta v, q \rangle  + \int_0^\infty g(s)\langle A^{\frac{\alpha}{2}}(u - \eta_s), A^{\frac{\alpha}{2}}\eta \rangle \mathrm{d}s \Bigg\} \notag\\
			&= -\int_0^\infty g(s)\langle A^{\frac{\alpha}{2}}\eta_s, A^{\frac{\alpha}{2}}\eta \rangle \mathrm{d}s= \frac{1}{2}\int_0^\infty g'(s)\|A^{\frac{\alpha}{2}}\eta\|^2 \mathrm{d}s \le 0,		
		\end{align}
		which shows that the operator $\mathcal{A}$ is dissipative.
		
		Next, we prove that $0 \in \rho(\mathcal{A})$.
		For any $F = (f_1, f_2, z_1, z_2, z_3)^\top \in \mathcal{H}$, we consider the equation
		\(
		\mathcal{A}(v,u,p,q,\eta)^\top = F,
		\)
		which is equivalent to the system
		\begin{align}
			u &= f_1\in D(A^{1/2}), \label{com1}\\
			-\alpha_1 A v + k A p + \zeta A^\alpha v - \int_0^\infty g(s) A^\alpha \eta(s) \mathrm{d}s &= f_2\in H, \label{com2}\\
			q &= z_1\in D(A^{1/2}), \label{com3}\\
			-\alpha_2 A p + k A^\beta v &= z_2\in H, \label{com4}\\
			u - \eta_s(s) &= z_3\in L^2_g.\label{com5}
		\end{align}
		From \eqref{com1}  and \eqref{com5}, we get $\eta_s(s)=f_1-z_3(s)$. Since $f_1\in D(A^{1/2})\subset D(A^{\frac\alpha2})$ and $z_3\in L^2_g$,  it is clear that
		\begin{equation}\label{etsdi}
			\eta_s\in L_g^2.
		\end{equation}
		Since $\eta(0)=0$, integrating $\eta_s(s)=f_1-z_3(s)$ yields
		\[
		\eta(s)=\int_0^s \big[f_1-z_3(r)\big]\mathrm{d}r = f_1 s - \int_0^s z_3(r)\mathrm{d}r.
		\]
		By assumption (A1) and using the Cauchy-Schwarz' and Young's inequalities, it follows that
		\begin{align*}
			\int_{0}^{\infty} g(s)\|A^{\frac{\alpha}{2}}\eta\|^2 ds
			&\leq -\frac{1}{k_{1}} \int_{0}^{\infty} g'(s)\|A^{\frac{\alpha}{2}}\eta\|^2 ds =\frac{2}{k_{1}} \operatorname{Re}\left(\int_{0}^{\infty} g(s)\langle A^{\frac{\alpha}{2}}\eta, A^{\frac{\alpha}{2}}\eta_{s}\rangle ds\right) \\
			&= \frac{2}{k_{1}} \int_{0}^{\infty} g(s)\|A^{\frac{\alpha}{2}}\eta\|
			\|A^{\frac{\alpha}{2}}\eta_{s}\| ds \leq \frac{1}{2} \int_{0}^{\infty} g(s)\|A^{\frac{\alpha}{2}}\eta\|^2 ds
			+\frac{2}{k_{1}^2} \int_{0}^{\infty} g(s)\|A^{\frac{\alpha}{2}}\eta_{s}\|^2 ds.
		\end{align*}
		Thus, by \eqref{etsdi},
		\[
		\int_{0}^{\infty} g(s)\|A^{\frac{\alpha}{2}} \eta\|^{2} d s
		\leq \frac{4}{k_{1}^{2}} \int_{0}^{\infty} g(s)\|A^{\frac{\alpha}{2}} \eta_{s}\|^{2} d s< \infty,
		\]
		Therefore, we obtain
		\begin{equation}\label{etsdi2}
			\eta\in L_g^2.
		\end{equation}
		In view of \eqref{etsdi} and \eqref{etsdi2}, we obtain
		\begin{equation}\label{etsdi3}
			\eta\in \mathcal{M}.
		\end{equation}
		
		The equations \eqref{com2} and \eqref{com4} can be transformed into the following ones
		\begin{align}
			\alpha_1 A v - kA p - \zeta A^\alpha v &= -f_2 - \displaystyle\int_{0}^{\infty} g(s) A^\alpha \eta(s) \, \mathrm{d}s, \label{ws1}\\
			\alpha_2 A p - kA^\beta v &= -z_2\label{ws2}.
		\end{align}
		We now show \eqref{ws1} and \eqref{ws2} admit a weak solution $(v,p)\in \mathcal{V}\triangleq{D}(A^{1/2})\times {D}(A^{1/2})$, which means  for any $(\varphi,\psi)\in\mathcal{V}$ there holds
		\begin{align}
			\alpha_1 \langle A^{\frac{1}{2}}v, A^{\frac{1}{2}}\varphi \rangle
			- k \langle A^{\frac{1}{2}}p, A^{\frac{1}{2}}\varphi \rangle
			- \zeta \langle A^{\frac{\alpha}{2}}v, A^{\frac{\alpha}{2}}\varphi \rangle
			&= -\langle f_2, \varphi \rangle-\int_0^\infty\langle A^{\frac\alpha 2}\eta(s), A^{\frac\alpha 2}\varphi \rangle ds,\label{alpha1} \\
			\alpha_2 \langle A^{\frac{1}{2}}p, A^{\frac{1}{2}}\psi \rangle
			- k \langle A^{\frac{\beta}{2}}v, A^{\frac{\beta}{2}}\psi \rangle
			&= -\langle z_2, \psi \rangle.\label{alpha2}
		\end{align}
		
		Note $f_2,z_2\in H$, $\eta\in L^2_g$ and $\alpha\le1$, the linear operator $G$ given by
		\begin{equation}\label{G}
			G(\varphi,\psi)\triangleq-\langle f_2, \varphi \rangle-\int_0^\infty\langle A^{\frac\alpha 2}\eta(s), A^{\frac\alpha 2}\varphi \rangle ds-\langle z_2, \psi \rangle
		\end{equation}
		is a bounded linear functional on $\mathcal{V}$, and we have, for some constant $C>0$,
		\begin{align*}
			|G(\varphi,\psi)|
			&\leq \|f_2\|\|\varphi\| + \int_0^\infty g(s)\|A^{\frac{\alpha}{2}}\eta\|\|A^{\frac{\alpha}{2}}\varphi\|ds + \|z_2\|\|\psi\| \\
			&\leq C\|F\|_{\mathcal{H}} \|(\varphi,\psi)\|_{\mathcal{V}} + \|A^{\frac{\alpha}{2}}\varphi\|\sqrt{\zeta}\|\eta\|_{L_g^2} \\
			&\leq  C(\|F\|_{\mathcal{H}}+\sqrt{\zeta}\|\eta\|_{L_g^2}) \|(\varphi,\psi)\|_{\mathcal{V}},
		\end{align*}
		where we have used $\|A^{\frac{\alpha}{2}}\varphi\|\le C\|A^{\frac{1}{2}}\varphi\|\le C\|(\varphi,\psi)\|_{\mathcal{V}}$.
		
		By adding the left-hand sides of \eqref{ws1} and \eqref{ws2} we get a linear operator $B$ defined on $\mathcal{V}\times\mathcal{V}$:
		\[
		B((v,p),(\varphi,\psi)) \triangleq\alpha_1 \langle A^{1/2}v,A^{1/2}\varphi\rangle
		- k\langle A^{\frac{1}{2}}p,A^{\frac{1}{2}}\varphi\rangle
		- k\langle A^{\frac{\beta}{2}}v,A^{\frac{\beta}{2}}\psi\rangle \\
		+\alpha_2 \langle A^{1/2}p,A^{1/2}\psi\rangle
		- \zeta\langle A^{\frac{\alpha}{2}}v,A^{\frac{\alpha}{2}}\varphi\rangle.
		\]
		It is obvious that $B(\cdot,\cdot)$ bounded, and by \eqref{coer} in the proof of Theorem \ref{thmequinorm}, we have $B(\cdot,\cdot)$ is coercive. By Lax-Milgram theorem \cite{MR2759829}, there exists a unique solution $(v,p)\in \mathcal{V}$, satisfying $B((v,p),(\varphi,\psi)) = G(\varphi,\psi)$, i.e., \eqref{ws1} and \eqref{ws2} admits a weak solution $(v,p)\in \mathcal{V}$. Moreover, in view of \eqref{com1}-\eqref{com4} and \eqref{etsdi3}, we get there exists $(v,u,p,q,\eta)^\top \in D(\mathcal{A})$ such that \(
		\mathcal{A}(v,u,p,q,\eta)^\top = F,
		\) i.e., the range of $\mathcal{A}$ is the whole $\mathcal{H}$.
		
		So to show $0\in\rho(\mathcal{A})$, we only need to show there exists $C>0$ such that
		\begin{equation}\label{yjgj}
			\|(v,u,p,q,\eta)^\top\|_{\mathcal{H}} \leq C\|F\|_{\mathcal{H}},
		\end{equation}where $(v,u,p,q,\eta)^\top$ is the solution of \(
		\mathcal{A}(v,u,p,q,\eta)^\top = F
		\) with $F = (f_1, f_2, z_1, z_2, z_3)^\top \in \mathcal{H}$.
		Then we have, for any $\varepsilon>0$, there exists a positive constant $k_\varepsilon$ such that
		\begin{align*}
			B((v,p),(v,p)) &= \alpha_1\|A^{1/2}v\|^2 - \zeta\|A^{\frac{\alpha}{2}}v\|^2 - k\operatorname{Re}\langle A^{\frac{1}{2}}p,A^{\frac{1}{2}}v\rangle-k\operatorname{Re}\langle A^{\frac{\beta}{2}}v,A^{\frac{\beta}{2}}p\rangle
			+ \alpha_2\|A^{1/2}p\|^2 \\
			&= G((v,p))\leq \|f_2\|\|v\| + \int_0^\infty g(s)\|A^{\frac{\alpha}{2}}\eta\|\|A^{\frac{\alpha}{2}}v\|ds + \|z_2\|\|p\| \\
			&\leq C\|F\|_{\mathcal{H}}\|A^{1/2}v\| + C\|A^{1/2}v\|\sqrt{\zeta}\|\eta\|_{L_g^2} + C\|F\|_{\mathcal{H}}\|A^{1/2}p\| \\
			&\leq \varepsilon\left(\|A^{1/2}v\|^2+\|A^{1/2}p\|^2\right) + k_\varepsilon\|F\|_{\mathcal{H}}^2.
		\end{align*}
		Since by \eqref{coer} in the proof of Theorem \ref{thmequinorm}, $B((v,p),(v,p))\ge\delta\left(\|A^{1/2}v\|^2+\|A^{1/2}p\|^2\right)$, where $\delta=\min\{\delta_1,\delta_0\}$, by choosing $\varepsilon=\delta/2$, we get
		$$\|A^{1/2}v\|^2+\|A^{1/2}p\|^2\leq C \|F\|_{\mathcal{H}}^2$$ for some constant $C>0$.
		Moreover, it is obvious that, there exists a positive constant $C$
		\[
		\|u\| = \|f_1\| \leq C\|A^{\frac{1}{2}}f_1\|
		\leq C\|F\|_{\mathcal{H}},\quad
		\|q\| = \|z_1\| \leq C\|A^{\frac{1}{2}}z_1\|
		\leq C\|F\|_{\mathcal{H}}.
		\]
		By assumption (A1) and $\eta_s(s)=f_1-z_3(s)$, we obtain
		\begin{align*}
			\|\eta\|_{L_g^2}^2= \int_0^\infty g(s) \| A^{\frac{\alpha}{2}} \eta \|^2 ds
			&\leq \frac{2}{k_1} \left| \int_0^\infty g(s) \langle A^{\frac{\alpha}{2}} \eta_s, A^{\frac{\alpha}{2}} \eta \rangle ds \right| \\
			&= \frac{2}{k_1} \left| \int_0^\infty g(s) \langle A^{\frac{\alpha}{2}} f_1, A^{\frac{\alpha}{2}} \eta \rangle ds \right|
			+ \frac{2}{k_1} | \int_0^\infty g(s) \langle A^{\frac{\alpha}{2}} z_3, A^{\frac{\alpha}{2}} \eta \rangle ds | \\
			&\leq \frac{2}{k_1} \int_0^\infty g(s) \| A^{\frac{\alpha}{2}} f_1 \| \| A^{\frac{\alpha}{2}} \eta \| ds
			+ \frac{2}{k_1} \int_0^\infty g(s) \| A^{\frac{\alpha}{2}} z_3 \| \| A^{\frac{\alpha}{2}} \eta \| ds \\
			&\leq \frac{2}{k_1} \sqrt{\zeta} \| A^{\frac{\alpha}{2}} f_1 \| \| \eta \|_{L^2_g}
			+ \frac{2}{k_1} \| \eta \|_{L^2_g} \| z_3 \|_{L^2_g}\\
			&\le\frac12\|\eta\|_{L^2_g}^2+\frac{4}{k_1^2}\zeta\| A^{\frac{\alpha}{2}} f_1 \|^2+\frac{4}{k_1^2}\| z_3 \|_{L^2_g}^2,
		\end{align*}
		so there exists a constant $C>0$ such that
		$$\|\eta\|_{L_g^2}^2\le \frac{8}{k_1^2}\zeta\| A^{\frac{\alpha}{2}} f_1 \|^2+\frac{8}{k_1^2}\| z_3 \|_{L^2_g}^2\le C\|F\|_{\mathcal{H}}^2.$$
		Thus the above analysis shows \eqref{yjgj} holds. The proof is completed.
	\end{proof}
	\section{Polynomial stability}\label{sec3}
The following characterization for polynomial stability of contraction $\mathcal{C}_0$-semigroups originates from Borichev and Tomilov \cite{MR2606945}.
\begin{theorem}\label{thmBT}
Assume $i\mathbb{R}\subset \rho(\mathcal{A})$ with $\rho(\mathcal{A})$ being the resolvent set of $\mathcal{A}$. Then the estimate
\[
\|S(t)X_0\|_{\mathcal{H}} \le \mathcal{M}\|X_0\|_{D(\mathcal{A})} \,t^{-\frac1\omega},\qquad \forall  t>0, X_0 \in D(\mathcal{A}),\;
\]
holds for some constant $\mathcal{M}>0$ and exponent $\omega>0$ if and only if
\begin{align}\label{eq:lemma1}
\sup_{\lambda\in\mathbb{R}} |\lambda|^{-\omega} \left\| (i\lambda I - \mathcal{A})^{-1} \right\|_{\mathcal{L}(\mathcal{H})} < \infty,
\end{align}
where $\left\| (i\lambda I - \mathcal{A})^{-1} \right\|_{\mathcal{L}(\mathcal{H})}$ denotes the norm of the operator $(i\lambda I - \mathcal{A})^{-1}$.
\end{theorem}
	\begin{theorem}\label{thmdecay}
		Assume that conditions (A1) and (A2) hold. Then the semigroup \(S(t)\) generated by the system \eqref{modelmain} is polynomially stable with decay rate \(t^{-\frac{1}{4-2\beta-2\alpha}}\) for smooth initial data. Specifically, for any initial state \(X_0=(v_0, v_1, p_0, p_1, \eta_0)^\top \in D(\mathcal{A})\), the corresponding solution $X(t)=(v,u,p,q,\eta)^\top=S(t)X_0$ satisfies
		\[
		\|X(t)\|_{\mathcal{H}} \leq C t^{-\frac{1}{4-2\beta-2\alpha}} \|X_0\|_{D(\mathcal{A})}, \quad \forall t \geq 1,
		\]
		where \(C > 0\) is a constant independent of \(t\) and the initial data and $\|X_0\|_{D(\mathcal{A})}$.
	\end{theorem}
	\begin{proof}
		Throughout our proofs, all limits are taken as $n\to\infty$.
		
\noindent \textbf{First, we show that ${i\mathbb{R} \subset \rho(\mathcal{A})}$.} By contradiction, suppose that $i\mathbb{R} \not\subset \varrho(\mathcal{A})$. Thus, due to $0 \in \varrho(\mathcal{A})$ and $\varrho(\mathcal{A})$ is an open set, we have
		\[
		0 < \hat{\lambda} < \infty,
		\]
		in which
		\[
		\hat{\lambda} := \sup\left\{ R > 0 : [-Ri, Ri] \subset \varrho(\mathcal{A}) \right\},
		\]
		
		Thanks to Banach-Steinhaus theorem, there exist sequences $X_n = (v_n, u_n, p_n, q_n, \eta_n) \in {D}(\mathcal{A})$ with $\|X_n\|_{\mathcal{H}} = 1$
		and $|\lambda_n| < \hat{\lambda}$, $\lambda_n \to \hat{\lambda}$ such that
		\begin{equation}\label{cov10}
			(i\lambda_n - \mathcal{A})X_n \equiv (f_n^1, f_n^2, z_n^1, z_n^2, z_n^3) \to 0 \quad \text{in } \mathcal{H},
		\end{equation}
		specifically,
		\begin{align}
			i\lambda_n v_n - u_n &= f_n^{1}\to0 \quad \text{in } {D}({A}^{\frac{1}{2}}), \label{gp1}\\
			i\lambda_n u_n - \left(-\alpha_1 A v_n + k A p_n + \zeta A^\alpha v_n - \int_0^\infty g(s) A^\alpha \eta_n\mathrm{d}s\right) &= f_n^{2}\to0 \quad \text{in } {H}, \label{gp2}\\
			i\lambda_n p_n - q_n &= z_n^{ 1}\to0 \quad \text{in } {D}({A}^{\frac{1}{2}}), \label{gp3}\\
			i\lambda_n q_n - \left(-\alpha_2 A p_n + k A^\beta v_n\right) &= z_n^{2}\to0 \quad \text{in } {H}, \label{gp4}\\
			i\lambda_n \eta_n - (u_n - \eta_{n,s}) &= z_n^{3}\to0 \quad \text{in } L_g^2.\label{gp5}
		\end{align}
		
		Next we verify
		\begin{equation}\label{gp}
			\|{A}^{\frac12}v_n\|,\ \|u_n\|,\ \|{A}^{\frac12}p_n\|,\ \|q_n\|,\ \|\eta_n\|_{L_g^2} = o(1)
		\end{equation}
		to get a contraction, where $g_n=o(1)$ means $g_n\to0$. In fact, if \eqref{gp} holds, by Theorem \ref{thmequinorm}, $$\|X_n\|_{\mathcal{H}}^2\le\kappa_1\left(\|A^{1/2}v_n\|^2+\|u_n\|^2+\|A^{1/2}p_n\|^2+\|q_n\|^2+\|\eta_n\|_{L^2_g}^2\right)\to0$$
		as $n\to\infty$, which contradicts $\|X_n\|_{\mathcal{H}}=1$.
		
	{\bf \noindent(1) $\bm{\|\eta_n\|_{L_g^2} =o(1)}$.}
		By \eqref{dissipative} and \eqref{cov10}, we have
		\[
		\int_0^\infty g'(s)\|A^{\frac{\alpha}{2}}\eta_n\|^2 \mathrm{d}s
		= 2\mathrm{Re}\langle \mathcal{A}X_n,X_n \rangle
		= -2\mathrm{Re}\langle \mathrm{i}\lambda_n X_n - \mathcal{A}X_n,X_n \rangle \to 0.
		\]
		Moreover, it follows the assumption (A2) that
		\[k_1\|\eta_n\|_{L^2_g}^2=k_1\int_0^\infty g(s)\|A^{\frac{\alpha}{2}}\eta_n(s)\|^2 \mathrm{d}s\le-\int_0^\infty g'(s)\|A^{\frac{\alpha}{2}}\eta_n\|^2 \mathrm{d}s.\]
		Then $\|\eta_n\|_{L_g^2} =o(1)$ follows from the above two relations.
		
		{\bf \noindent(2) $\bm{\|u_n\|=o(1)}$.}
		Taking the $L_g^2$-inner product of equation \eqref{gp5} with $u_n$ yields
		\begin{equation}\label{cov2}
			\int_{0}^{\infty} g(s)\langle A^{\frac{\alpha}{2}}(\mathrm{i}\lambda_{n}\eta_{n}-u_{n}+\eta_{n,s}),A^{\frac{\alpha}{2}}u_{n}\rangle\mathrm{d}s \le\|z^3_n\|_{L^2_g}\|u_n\|_{L^2_g}\to 0,
		\end{equation}
		where we have used $\|z^3_n\|_{L^2_g}\to0$ and $\|u_n\|_{L^2_g}\le\sqrt\zeta\|A^{\frac\alpha2}u_n\|\le C\|X_n\|_{\mathcal{H}}=C$ for some constant $C>0$.
		
		Since $\|\eta_n\|_{L_g^2} =o(1)$, by the boundedness of $\| A^{\frac{\alpha}{2}} u_n \|$ and $|\lambda_n|$ with respect to $n$, we obtain
		\begin{align*}
			\left|\int_{0}^{\infty} g(s)\langle A^{\frac{\alpha}{2}}\mathrm{i}\lambda_{n}\eta_{n},A^{\frac{\alpha}{2}}u_{n}\rangle\mathrm{d}s\right|= |\lambda_{n}|\left|\int_{0}^{\infty} g(s)\langle A^{\frac{\alpha}{2}}\eta_{n},A^{\frac{\alpha}{2}}u_{n}\rangle\mathrm{d}s| \leq |\lambda_{n}\right| \|A^{\frac{\alpha}{2}}u_{n}\|  \sqrt{\zeta} \|\eta_{n}\|_{L_{g}^{2}} \to 0,
		\end{align*}
		and by the assumption (A2),
		\begin{align*}
			\left|\int_{0}^{\infty} g(s)\langle A^{\frac{\alpha}{2}}\eta_{n,s},A^{\frac{\alpha}{2}}u_{n}\rangle\mathrm{d}s\right|&= \left|-\int_{0}^{\infty} g'(s)\langle A^{\frac{\alpha}{2}}\eta_{n},A^{\frac{\alpha}{2}}u_{n}\rangle\mathrm{d}s\right| \\
			&\leq k_{0}\int_{0}^{\infty} g(s)\|A^{\frac{\alpha}{2}}\eta_{n}\|\mathrm{d}s \|A^{\frac{\alpha}{2}}u_{n}\| \leq k_{0}\|A^{\frac{\alpha}{2}}u_{n}\|\sqrt{\zeta}\|\eta_{n}\|_{L_{g}^{2}} \to 0.
		\end{align*}
		Combining the above two relations with 	\eqref{cov2}, we get
		\[
		\left|\int_{0}^{\infty} g(s)\langle A^{\frac{\alpha}{2}}u_{n},A^{\frac{\alpha}{2}}u_{n}\rangle\mathrm{d}s\right|
		= \zeta\|A^{\frac{\alpha}{2}}u_{n}\|^{2}\to0,
		\]
		which implies $\|u_n\|=o(1)$.
		
		{\bf \noindent(3) $\bm{\|A^{\frac{1}{2}} v_n\|=o(1)}$.}	
		Substituting \eqref{gp1} into \eqref{gp2} yields
		\begin{equation}\label{cov3}
			-\lambda_n^2 v_n + \alpha_1 A v_n - k A p_n - \zeta A^\alpha v_n + \int_0^\infty g(s) A^\alpha \eta_n \mathrm{d}s =i\lambda_nf_n^1+f_n^2\to 0 \quad \text{in } H.
		\end{equation}
		Substituting \eqref{gp3} into \eqref{gp4}  yields
		\begin{equation}\label{cov4}
			-\lambda_n^2 p_n + \alpha_2 A p_n - k A^\beta v_n =i\lambda_nz_n^1+z_n^2\to 0 \quad \text{in } H.
		\end{equation}
		Multiplying \eqref{cov4} by $\frac{k}{\alpha_2}$ and adding it to \eqref{cov3}, we obtain
		\begin{equation}\label{cov5}
			-\lambda_n^2 v_n + \alpha_1 A v_n - \zeta A^\alpha v_n + \int_0^\infty g(s) A^\alpha \eta_n \, ds - \frac{k}{\alpha_2} \lambda_n^2 p_n - \frac{k^2}{\alpha_2} A^\beta v_n \to 0
		\end{equation}
		Taking the $L^2$-inner product of \eqref{cov5} with $v_n$ yields
		\begin{align}\label{eq:inner_product_vn}
			&-\lambda_n^2 \|v_n\|^2 + \alpha_1 \|A^{\frac{1}{2}} v_n\|^2 - \zeta \|A^{\frac{\alpha}{2}} v_n\|^2 \notag \\
			&+ \int_0^\infty g(s) \langle A^{\frac{\alpha}{2}} \eta_n, A^{\frac{\alpha}{2}} v_n \rangle ds - \frac{k}{\alpha_2} \lambda_n^2 \langle p_n, v_n \rangle
			- \frac{k^2}{\alpha_2} \|A^{\frac{\beta}{2}} v_n\|^2 \to 0.
		\end{align}
		
		Since $\|u_n\|=o(1)$, it follows from \eqref{gp1} that
		\begin{equation}\label{vn1}
			\lambda_n^2 \|v_n\|^2 \to0.
		\end{equation}
		Since $\|\eta_n\|_{L_g^2}=o(1)$, and $\| A^{\frac{\alpha}{2}} v_n \|$ is bounded with respect to $n$, we obtain
		\begin{equation}\label{vn2}
			\left| \int_0^\infty g(s) \langle A^{\frac{\alpha}{2}} \eta_n, A^{\frac{\alpha}{2}} v_n \rangle ds \right| \leq \sqrt{\zeta} \|A^{\frac{\alpha}{2}} v_n\|\|\eta_n\|_{L_g^2} \to 0.
		\end{equation}
		By virtue of \eqref{gp1}, \eqref{gp3}, the boundedness of $\|q_n\|$, and $\|u_n\|,~\|f_n^1\|,~\|z_n^1\|=o(1)$ , we have
		\begin{equation}\label{vn3}
			\lambda_n^2 |\langle p_n, v_n \rangle|
			=  \langle q_n+z_n^1, u_n+f_n^1 \rangle\le\|q_n\|(\|u_n\|+\|f_n^1\|)+\|z_n^1\|(\|u_n\|+\|f_n^1\|) \to 0.
		\end{equation}
		By \eqref{eq:inner_product_vn}-\eqref{vn3}, we obtain
		\begin{equation}\label{vn4}
			\alpha_1 \| A^{\frac{1}{2}} v_n \|^2
			- \zeta \| A^{\frac{\alpha}{2}} v_n \|^2
			- \frac{k^2}{\alpha_2} \| A^{\frac{\beta}{2}} v_n \|^2
			\to 0.
		\end{equation}
As the proof of Theorem \ref{thmequinorm},
$$\alpha_1 \| A^{\frac{1}{2}} v_n \|^2
			- \zeta \| A^{\frac{\alpha}{2}} v_n \|^2
			- \frac{k^2}{\alpha_2} \| A^{\frac{\beta}{2}} v_n \|^2\ge\left(\alpha_1-\zeta\gamma_1^{\alpha-1}-\frac{k^2}{\alpha_2}\gamma_1^{\beta-1}\right)\|A^{1/2}v_n\|.$$
Since, by the Assumption (A2), $\alpha_1-\zeta\gamma_1^{\alpha-1}-\frac{k^2}{\alpha_2}\gamma_1^{\beta-1}>0$, we deduce from the above inequality and \eqref{vn4} that  $\| A^{\frac{1}{2}} v_n \|=o(1).$
		
		{\bf \noindent(4) $\bm{\|q_n\|=o(1)}$ and $\bm{\|A^{\frac{1}{2}} p_n\|=o(1)}$.} Since $\|p_n\|$ is bounded, taking the $H$-inner product of \eqref{cov5} with $p_n$, we obtain
		\begin{align}\label{qn1}
			&-\lambda_n^2 \langle v_n, p_n \rangle
			+ \alpha_1 \langle A^{\frac{1}{2}} v_n, A^{\frac{1}{2}} p_n \rangle
			- \zeta \langle A^{\frac{\alpha}{2}} v_n, A^{\frac{\alpha}{2}} p_n \rangle \notag \\
			&+ \int_0^\infty g(s) \langle A^{\frac{\alpha}{2}} \eta_n, A^{\frac{\alpha}{2}} p_n \rangle ds
			- \frac{k}{\alpha_2} \lambda_n^2 \|p_n\|^2
			- \frac{k^2}{\alpha_2} \langle A^{\frac{\beta}{2}} v_n, A^{\frac{\beta}{2}} p_n \rangle
			\to 0.
		\end{align}
		By \eqref{vn3}, we have
		\begin{equation}\label{qn2}
			-\lambda_n^2 \langle v_n, p_n \rangle = -\langle u_n, q_n \rangle \to 0.
		\end{equation}
		Since $\| A^{\frac{1}{2}} v_n \|=o(1)$ and $\alpha,\beta<1$, we have $\| A^{\frac{\alpha}{2}} v_n \|,~\| A^{\frac{\beta}{2}} v_n \|=o(1)$, then it follows from the  boundedness of $\|A^{\frac{1}{2}} p_n\|$ that
		\begin{equation}\label{qn3}
			\alpha_1 \langle A^{\frac{1}{2}} v_n, A^{\frac{1}{2}} p_n \rangle
			- \zeta \langle A^{\frac{\alpha}{2}} v_n, A^{\frac{\alpha}{2}} p_n \rangle
			- \frac{k^2}{\alpha_2} \langle A^{\frac{\beta}{2}} v_n, A^{\frac{\beta}{2}} p_n \rangle
			\to 0.
		\end{equation}
		Since $\|\eta_n\|_{L_g^2}=o(1)$, and $\| A^{\frac{\alpha}{2}} p_n \|$ is bounded with respect to $n$, we obtain
		\begin{equation}\label{qn4}
			\left| \int_0^\infty g(s) \langle A^{\frac{\alpha}{2}} \eta_n, A^{\frac{\alpha}{2}} p_n \rangle ds \right| \leq \sqrt{\zeta} \|A^{\frac{\alpha}{2}} p_n\|\|\eta_n\|_{L_g^2} \to 0.
		\end{equation}
Then \eqref{vn3}, \eqref{qn1}-\eqref{qn4} yield
\begin{equation}\label{qn5-1}
		\lambda_n^2 \|p_n\|^2 \to 0.
		\end{equation}
By \eqref{gp3}, we get $$\|q_n\|^2=\|i\lambda_np_n-z_n^1\|^2\le\lambda_n^2\|p_n\|^2+2|\lambda_n|\|p_n\|\|z_n^1\|+\|z_n^1\|^2.$$
Since $|\lambda_n|$ and $\|p_n\|$ are bounded, $\|z_n^1\|=o(1)$, we get from \eqref{qn5-1} that $\|q_n\|=o(1)$.
		
Since $\|p_n\|$ is bounded, taking the $H$-inner product of \eqref{cov4} with $p_n$, we have
		\begin{equation}\label{pn1}
			-\lambda_n^2 \|p_n\|^2 + \alpha_2 \|A^{\frac{1}{2}}p_n\|^2 - k \langle A^{\frac{\beta}{2}}v_n, A^{\frac{\beta}{2}}p_n \rangle \to 0.
		\end{equation}
Since $\| A^{\frac{\beta}{2}} v_n \|=o(1),$ the boundedness of $\|A^{\frac{\beta}{2}} p_n\|$ and \eqref{qn5-1}, imply
$$\left|\lambda_n^2 \|p_n\|^2+k \langle A^{\frac{\beta}{2}}v_n, A^{\frac{\beta}{2}}p_n \rangle\right|\le\lambda_n^2 \|p_n\|^2+k\|A^{\frac{\beta}{2}}v_n\|\|A^{\frac{\beta}{2}}p_n\|\to0.$$
Then by \eqref{pn1}, we obtain $\|A^{\frac{1}{2}} p_n\|=o(1)$.
		
		From the above (1)-(4), we obtain $\|X_n\|_{\mathcal{H}}=o(1) $, which contradicts $\|X_n\|_{\mathcal{H}} = 1$.
		Therefore $i\mathbb{R} \subset \rho(\mathcal{A})$.

{\bf Second, we prove
\begin{equation}\label{secmain}
\limsup_{|\lambda| \to \infty} |\lambda|^{-{(4 - 2\beta - 2\alpha)}} \| (\mathrm{i}\lambda -\mathcal{A})^{-1} \| < \infty.
\end{equation}
}

By contradiction, if \eqref{secmain} is not correct, then there exists
		\[
		X_n = (v_n, u_n, p_n, q_n,\eta_n) \in D(\mathcal{A}),\quad \|X_n\|_\mathcal{H}= 1,
		\]
		and $\lambda_n \to \infty$ such that
		\begin{equation}\label{sj}
			\lambda_n^{4 - 2\beta - 2\alpha} (\mathrm{i}\lambda_n - \mathcal{A})X_n = (f_n^1, f_n^2, z_n^1, z_n^2, z_n^3) \to 0 \quad \text{in } \mathcal{H},
		\end{equation}
	specifically
		\begin{align}
			\lambda_n^{4 - 2\beta - 2\alpha} (\mathrm{i}\lambda_n v_n - u_n) &= f_n^1\to0 \quad \text{in } {D}({A}^{\frac{1}{2}}),\label{sj1} \\
			\lambda_n^{4 - 2\beta - 2\alpha} \left( \mathrm{i}\lambda_n u_n + \alpha_1 A v_n - k A p_n -\zeta A^\alpha v_n+ \int_0^\infty g(s) A^\alpha\eta_n \, \mathrm{d}s \right) &= f_n^2\to0 \quad \text{in } {H}, \label{sj2}\\
			\lambda_n^{4 - 2\beta - 2\alpha} (\mathrm{i}\lambda_n p_n - q_n) &= z_n^1\to0 \quad \text{in } {D}({A}^{\frac{1}{2}}), \label{sj3} \\
			\lambda_n^{4 - 2\beta - 2\alpha} \left( \mathrm{i}\lambda_n q_n + \alpha_2 A p_n - k A^\beta v_n \right) &= z_n^2\to0 \quad \text{in } {H},  \label{sj4}\\
			\lambda_n^{4 - 2\beta - 2\alpha} (\mathrm{i}\lambda_n \eta_n - u_n + \eta_{n,s}) &= z_n^3\to0 \quad \text{in } L_g^2.\label{sj5}
		\end{align}

{\bf \noindent(1) $\bm{\|\eta_n\|_{L_g^2}=|\lambda_n|^{\alpha + \beta - 2}o(1)}$.}
		By \eqref{dissipative}, \eqref{sj} and $\|X_n\|_{\mathcal{H}}=1$, we have
		\begin{align*}
			|\lambda_n|^{4 - 2\beta - 2\alpha} \int_0^\infty g'(s) \|A^{\frac{\alpha}{2}} \eta_n\|^2 ds
			&= |\lambda_n|^{4 - 2\beta - 2\alpha}  \left(-2\operatorname{Re} \langle i\lambda_n X_n -\mathcal{A} X_n, X_n \rangle\right) \\
			&= -2\operatorname{Re} \langle |\lambda_n|^{4 - 2\beta - 2\alpha} (i\lambda_n X_n - \mathcal{A} X_n), X_n \rangle \to 0,
		\end{align*}
which, together with the Assumption (A1), i.e.,
		\[
-|\lambda_n|^{4 - 2\beta - 2\alpha} \int_0^\infty g'(s) \|A^{\frac{\alpha}{2}} \eta_n\|^2 ds
		\geq k_1 |\lambda_n|^{4 - 2\beta - 2\alpha} \|\eta_n\|_{L_g^2}^2,
		\]
implies $|\lambda_n|^{4 - 2\beta - 2\alpha}\|\eta_n\|_{L_g^2}^2 =o(1)$,  i.e.,
		\begin{equation}\label{sj6}
			\|\eta_n\|_{L_g^2} =|\lambda_n|^{\alpha + \beta - 2} o(1).
		\end{equation}
		{\bf \noindent(2) $\bm{\|u_n\|=o(1)}$.} By \eqref{sj5}, we have
		\begin{equation}\label{un1}
			i\lambda_n \eta_n - u_n + \eta_{n,s} = \frac{z_n^3}{\lambda_n^{4 - 2\beta - 2\alpha}}.
		\end{equation}
Taking the $L_g^2$-inner product of \eqref{un1} with $u_n$, we obtain
		\begin{align}\label{un2}
			&\int_0^\infty i\lambda_n g(s)\langle A^{\frac{\alpha}{2}}\eta_n, A^{\frac{\alpha}{2}}u_n \rangle \, ds
			- \zeta\|A^{\frac{\alpha}{2}}u_n\|^2  \nonumber \\
			&+ \int_0^\infty g(s)\langle A^{\frac{\alpha}{2}}\eta_{n,s}, A^{\frac{\alpha}{2}}u_n \rangle \, ds
			= \frac{1}{\lambda_n^{4-2\beta-2\alpha}}
			\int_0^\infty g(s)\langle A^{\frac{\alpha}{2}}z_n^3, A^{\frac{\alpha}{2}}u_n \rangle \, ds.
		\end{align}
		For the first term on the left-hand side of \eqref{un2}, by applying H\"older's inequality, we obtain
\[\left| \int_0^\infty i\lambda_n g(s) \langle A^{\frac{\alpha}{2}} \eta_n, A^{\frac{\alpha}{2}} u_n \rangle ds \right|
			\leq |\lambda_n| \| A^{\frac{\alpha}{2}} u_n \| \sqrt{\zeta} \left\| \eta_n \right\|_{L_g^2}.\]
Then it follows from \eqref{sj6} that
		\begin{equation}\label{un3}
			\left| \int_0^\infty i\lambda_n g(s) \langle A^{\frac{\alpha}{2}} \eta_n, A^{\frac{\alpha}{2}} u_n \rangle ds \right|
			= |\lambda_n|^{\alpha + \beta - 1} \| A^{\frac{\alpha}{2}} u_n \| o(1).
		\end{equation}
		For the third term on the left-hand side of \eqref{un2}, by applying integration by parts, using assumption (A1), we obtain
\begin{align*}
\left| \int_0^\infty g(s) \langle A^{\frac{\alpha}{2}} \eta_{n,s}, A^{\frac{\alpha}{2}} u_n \rangle ds\right |
			&= \frac12\left|\int_0^\infty g'(s) \langle A^{\frac{\alpha}{2}} \eta_n, A^{\frac{\alpha}{2}} u_n \rangle ds \right|  \\
			&\leq \frac12k_0 \left| \int_0^\infty g(s) \langle A^{\frac{\alpha}{2}} \eta_n, A^{\frac{\alpha}{2}} u_n \rangle ds \right|  \\
			&\leq \frac12k_0 \| A^{\frac{\alpha}{2}} u_n \| \sqrt{\zeta} \left\| \eta_n \right\|_{L_g^2} .
\end{align*}
Then it follows from \eqref{sj6} that
		\begin{align}\label{un4}
		\left| \int_0^\infty g(s) \langle A^{\frac{\alpha}{2}} \eta_{n,s}, A^{\frac{\alpha}{2}} u_n \rangle ds\right |= |\lambda_n|^{\alpha+\beta-2}  \| A^{\frac{\alpha}{2}} u_n \| o(1)
		\end{align}
		 For the right-hand side of \eqref{un2}, by applying H\"older's inequality we obtain
\[\left| \frac{1}{\lambda_n^{4-2\beta-2\alpha}} \int_0^\infty g(s) \langle A^{\frac{\alpha}{2}} z_n^3, A^{\frac{\alpha}{2}} u_n \rangle ds \right|
			\le \frac{1}{|\lambda_n|^{4-2\beta-2\alpha}} \sqrt{\zeta} \| A^{\frac{\alpha}{2}} u_n \|\|z_n^3\|_{L_g^2}.\]
Since $\|z_n^3\|_{L_g^2}=o(1)$, it follows
		\begin{equation}\label{un5}
			\left| \frac{1}{\lambda_n^{4-2\beta-2\alpha}} \int_0^\infty g(s) \langle A^{\frac{\alpha}{2}} z_n^3, A^{\frac{\alpha}{2}} u_n \rangle ds \right|
			=\frac{1}{|\lambda_n|^{4-2\beta-2\alpha}}\| A^{\frac{\alpha}{2}} u_n \| o(1).
		\end{equation}
Since by \eqref{un2},
\begin{align*}
 \zeta\|A^{\frac{\alpha}{2}}u_n\|^2\le&\left| \int_0^\infty i\lambda_n g(s) \langle A^{\frac{\alpha}{2}} \eta_n, A^{\frac{\alpha}{2}} u_n \rangle ds \right|+\left| \int_0^\infty g(s) \langle A^{\frac{\alpha}{2}} \eta_{n,s}, A^{\frac{\alpha}{2}} u_n \rangle ds\right |\\
 &+\left| \frac{1}{\lambda_n^{4-2\beta-2\alpha}} \int_0^\infty g(s) \langle A^{\frac{\alpha}{2}} z_n^3, A^{\frac{\alpha}{2}} u_n \rangle ds \right|,
\end{align*}
it follows from \eqref{un3}-\eqref{un5} that
\[\|A^{\frac{\alpha}{2}}u_n\|^2=\left( |\lambda_n|^{\alpha + \beta - 1} +|\lambda_n|^{\alpha+\beta-2} +\frac{1}{|\lambda_n|^{4-2\beta-2\alpha}}\right)\| A^{\frac{\alpha}{2}} u_n \| o(1).\]
Noting $\lambda_n\to\infty$ and $\alpha,\beta\le1$, we obtain
		\begin{align}\label{un6}
			\| A^{\frac{\alpha}{2}} u_n \| = |\lambda_n|^{\alpha + \beta - 1} o(1).
		\end{align}
	Taking the $H$-inner product of \eqref{sj5} with $A^{\alpha-1}u_n$, then multiplying by $g(s)$ and integrating with respect to $s$ from $0$ to $\infty$ yields that
		\begin{align}\label{2un1}
			\zeta \| A^{\frac{\alpha-1}{2}} u_n \|^2
			&= -\int_0^\infty  g(s) \langle A^{\frac{\alpha-1}{2}} \eta_n, A^{\frac{\alpha-1}{2}}i \lambda_n u_n \rangle \, ds
			+ \int_0^\infty g(s) \langle A^{\frac{\alpha-1}{2}} \eta_{n,s}, A^{\frac{\alpha-1}{2}} u_n \rangle \, ds \nonumber \\
			&\quad - \frac{1}{\lambda_n^{4-2\beta-2\alpha}} \int_0^\infty g(s) \langle A^{\frac{\alpha-1}{2}} z_n^3, A^{\frac{\alpha-1}{2}} u_n \rangle \, ds.
		\end{align}
Since $\alpha<1$, $ D(A^{{\alpha}/{2}})\hookrightarrow D(A^{\frac{2\alpha-1}{2}})\hookrightarrow D(A^{\frac{\alpha-1}{2}})$ continuously, then there exists three  constants \( C_1 >0\) , \( C_2>0 \)  and $C_3>0$such that
		\[\|A^{\frac{2\alpha-1}{2}}w\| \le C_1\|A^{\frac{\alpha}{2}}w\|,\quad
		\|A^{\frac{\alpha-1}{2}}w\| \le C_2\|A^{\frac{\alpha}{2}}w\|,\quad
		\|A^{\frac{\alpha-1}{2}}w\| \le C_3\|w\|~~~\forall w\in D(A^{{\alpha}/{2}}).\]
		Then for the second and third terms on the right-hand side of \eqref{2un1}, by similar calculations of \eqref{un4} and \eqref{un5}, we obtain
		\begin{align}\label{2un2}
			\left|\int_0^\infty g(s)\langle A^{\frac{\alpha-1}{2}}\eta_{n.s}, A^{\frac{\alpha-1}{2}}u_n\rangle ds\right| +\left| \frac{1}{\lambda_n^{4-2\beta-2\alpha}}\int_0^\infty g(s)\langle A^{\frac{\alpha-1}{2}}z_n^3, A^{\frac{\alpha-1}{2}}u_n\rangle ds\right|\notag\\
=\left(|\lambda_n|^{\alpha +\beta -2}+|\lambda_n|^{2\alpha +2\beta -4}\right)\|A^{\frac{\alpha-1}{2}}u_n\|o(1).
		\end{align}
Since by \eqref{sj2},
		\[i\lambda_n u_n = -\alpha_1 A v_n + k A p_n + \zeta A^\alpha v_n - \int_0^\infty g(s) A^\alpha \eta_n ds + \frac{f_n^2}{\lambda_n^{4-2\beta-2\alpha}},\]
for the first terms on the right-hand side of \eqref{2un1}, we have
		\begin{align*}
			&\left| \int_0^\infty g(s) \langle
			A^{\frac{\alpha-1}{2}} \eta_n, A^{\frac{\alpha-1}{2}}
			i\lambda_n u_n \rangle ds \right|
			= \left| \int_0^\infty g(s) \langle A^{\alpha-1} \eta_n,
			i\lambda_n u_n \rangle ds \right| \nonumber \\
			&\leq \alpha_1 \left| \int_0^\infty g(s) \langle
			A^{\frac{\alpha}{2}} \eta_n, A^{\frac{\alpha}{2}} v_n
			\rangle ds \right|
			+ k \left| \int_0^\infty g(s) \langle
			A^{\frac{\alpha}{2}} \eta_n, A^{\frac{\alpha}{2}} p_n
			\rangle ds \right|  \\
			&\quad + \zeta \left| \int_0^\infty g(s) \langle
			A^{\frac{2\alpha-1}{2}} \eta_n, A^{\frac{2\alpha-1}{2}} v_n
			\rangle ds \right|
			+ \zeta \left| \int_0^\infty g(s) \|
			A^{\frac{2\alpha-1}{2}} \eta_n \|^2 ds \right|  \\
			&\quad + \frac{1}{|\lambda_n|^{4-2\beta-2\alpha}} \left|
			\int_0^\infty g(s) \langle A^{\frac{\alpha-1}{2}}
			\eta_n, A^{\frac{\alpha-1}{2}} f_n^2 \rangle ds \right|\\
&  \leq \alpha_1 \sqrt{\zeta}\| A^{\frac{\alpha}{2}} v_n \|  \|\eta\|_{L^2_g}
			+ k \sqrt{\zeta}\| A^{\frac{\alpha}{2}} p_n \| \|\eta\|_{L^2_g} + \zeta^{\frac{3}{2}} C_1^2 \| A^{\frac{\alpha}{2}} v_n \| \|\eta\|_{L^2_g}
			+ \zeta C_1^2\|\eta\|_{L^2_g}^2\notag \\
			& \quad + C_2C_3 |\lambda_n|^{2\alpha +2\beta -4} \sqrt{\zeta} \, \left\| f_n^2 \right\|\|\eta\|_{L^2_g}.
		\end{align*}
	Then by \eqref{sj6} and the boundedness of $\| A^{\frac{\alpha}{2}} v_n \|$, $\| A^{\frac{\alpha}{2}} p_n \|$ and $\|f_n^2\|$, we obtain
		\begin{align}\label{2un4}
	\left| \int_0^\infty g(s) \langle A^{\frac{\alpha-1}{2}} \eta_n, A^{\frac{\alpha-1}{2}} i\lambda_n u_n \rangle ds \right|=\left(|\lambda_n|^{\alpha +\beta -2}+|\lambda_n|^{2\alpha +2\beta -4}+|\lambda_n|^{3\alpha +3\beta -6}\right)o(1).
		\end{align}
Noting $\lambda_n\to\infty$ and $\alpha,\beta\le1$, 	from \eqref{2un1}, \eqref{2un2}, \eqref{2un4}, we obtain
		\begin{equation}\label{2un5}
			\|A^{\frac{\alpha-1}{2}}u_n\|=|\lambda_n|^{\alpha +\beta -2}o(1).
		\end{equation}
Since by virtue of the interpolation inequality,
\[\|u_n\| \le\|A^{\frac{\alpha}{2}}u_n\|^{1-\alpha}\|A^{\frac{\alpha-1}{2}}u_n\|^{\alpha},\]
it follows from \eqref{un6} and \eqref{2un5} that
\[\|u_n\|=\left(|\lambda_n|^{\alpha + \beta - 1} o(1)\right)^{1-\alpha}\left(|\lambda_n|^{\alpha +\beta -2}o(1)\right)^{\alpha}=|\lambda_n|^{\beta-1}o(1).\]
Since $\beta<1$ and $\lambda_n\to\infty$, we get $	\| u_n \|=o(1)$.

By similar proofs of the steps (3) and (4) in the proof of $i\mathbb{R}\subset\rho(\mathcal{A})$, we obtain
		\begin{equation}\label{sum}
			\| A^{\frac{1}{2}} v_n \|=o(1), \quad \| q_n \|=o(1), \quad \| A^{\frac{1}{2}} p_n \|=o(1).
		\end{equation}
		
		Summing up the above (1), (2), and \eqref{sum} , we have derived that $\|X_n\|_\mathcal{H}=o(1)$, which contradicts $\|X_n\|_\mathcal{H}= 1$. The proof is completed.
\end{proof}
	\section{Optimal decay rate}\label{sec4}
The main result of this section is the following theorem:
	\begin{theorem}\label{thmoptimaldecay}
If the memory kernel decays exponentially and $\alpha_1\neq\alpha_2$, then the polynomial decay rate established in Theorem \ref{thmdecay} is sharp.
\end{theorem}
\begin{remark}\label{rem-sharpness}
Theorems \ref{thmdecay} and \ref{thmoptimaldecay} reveal a clear distinction regarding the optimality of the derived polynomial decay rate, depending on the relation between the coefficients $\alpha_1$ and $\alpha_2$. In the case $\alpha_1 \neq \alpha_2$, the decay exponent $\frac{1}{4-2\beta-2\alpha}$ obtained in Theorem \ref{thmdecay} is sharp, meaning that no faster algebraic decay rate can be achieved for smooth initial data. Conversely, when $\alpha_1 = \alpha_2$, our analysis only establishes an upper bound for the decay index $\delta$, for which the estimate
\begin{equation}\label{remarkeq1}
 \|S(t)X_0\|_{\mathcal{H}} \leq C t^{-\delta} \|X_0\|_{D(\mathcal{A})}, \qquad \forall t \ge 1,
\end{equation}
holds with
\[
\delta \le \frac{1}{3-\beta-2\alpha}.
\]
This indicates that the polynomial decay rate $t^{-\frac{1}{4-2\beta-2\alpha}}$ from Theorem \ref{thmdecay} is no longer sharp under the constraint $\alpha_1=\alpha_2$. There are two plausible explanations for this loss of sharpness:
\begin{enumerate}
\item The true optimal decay rate in the symmetric case $\alpha_1=\alpha_2$ is actually governed by the exponent $\frac{1}{3-\beta-2\alpha}$ rather than $\frac{1}{4-2\beta-2\alpha}$;
\item The decay estimate $t^{-\frac{1}{4-2\beta-2\alpha}}$ remains sharp even in the symmetric setting, yet the technical tools employed in the present work are insufficient to confirm its optimality for this particular parameter configuration.
\end{enumerate}

We accordingly raise the following {\bf open problem} for future investigation:

For the abstract magnetizable piezoelectric system \eqref{modelmain} under assumptions (A1)--(A2) with identical stiffness coefficients $\alpha_1=\alpha_2$, determine the exact sharp polynomial decay rate of solutions from smooth initial data lying in $D(\mathcal{A})$. More precisely, clarify whether the optimal decay exponent equals $\dfrac{1}{3-\beta-2\alpha}$ or $\dfrac{1}{4-2\beta-2\alpha}$.
\end{remark}
	\begin{proof}[Proof of Theorem \ref{thmoptimaldecay}]
		Let $g(s)=g(0)e^{-\delta s}$ with $g(0)>0$ and $\delta>0$ be an exponential form. We have $Ae_n = \gamma_n e_n$, where $\{\gamma_n\}_{n\in\mathbb{N}}$ denote the eigenvalues of $A$ and $\{e_n\}_{n\in\mathbb{N}}$ the associated eigenfunctions. Moreover, $\gamma_n\to\infty$ and $\|e_n\|=1$ for all $n\in\mathbb{N}$.

Since we have shown in Section \ref{sec3} that $i\mathbb{R}\subset\rho(\mathcal{A})$, let $ U=(v,u,p,q,\eta)\in D(\mathcal{A})$ be given by $U=(i\lambda I-\mathcal{A})^{-1}(0,0,0,-e_n,0)^\top$, i.e.,
		\begin{align}
			i\lambda v - u &= 0,\label{op1} \\
			i\lambda u + \alpha_1 A v - k A p - \zeta A^\alpha v + \int_0^\infty g(s)A^\alpha \eta\,ds &= 0 \label{op2},\\
			i\lambda p - q &= 0 \label{op3},\\
			i\lambda q + \alpha_2 A p - k A^\beta v &= -e_n,\label{op4} \\
			i\lambda \eta + \eta_s \label{op5}&= u.
		\end{align}
		From \eqref{op1} and \eqref{op5}, we have
		\[
		\eta_s = -i\lambda \eta + u = -i\lambda \eta + i\lambda v.
		\]
		Since $\eta(0)=0$, solving this ordinary differential equation, we get
		\begin{equation}\label{etas}
			\eta(s)=\int_0^s e^{-i\lambda(s-t)}i\lambda v\,dt = v(1-e^{-i\lambda s}).
		\end{equation}
		By \eqref{op1} -\eqref{op4} and \eqref{etas}, we obtain
		\begin{align}
			\lambda^2 v - \alpha_1 A v + k A p + \int_0^\infty g(s)e^{-i\lambda s}ds\,A^\alpha v = 0,\label{lamda} \\
			\lambda^2 p - \alpha_2 A p + k A^\beta v = e_n.\label{lamda1}
		\end{align}
		Let
\begin{equation}\label{vp}
  v=k_1 e_n,~p=k_2 e_n,~~~ k_1,k_2\in\mathbb{C}.
\end{equation}
		Substituting \eqref{vp} into \eqref{lamda} and \eqref{lamda1}, we have
		\[
		\begin{cases}
			\left(\lambda^2-\alpha_1\gamma_n+\displaystyle\int_0^\infty g(s)e^{-i\lambda s}ds\,\gamma_n^\alpha\right)k_1 + k\gamma_n k_2 = 0,\\
			\left(\lambda^2-\alpha_2\gamma_n\right)k_2 + k\gamma_n^\beta k_1 = 1,
		\end{cases}
		\]
		therefore
		\begin{align}\label{lamdak}
			\begin{pmatrix}
				\lambda^2 - \alpha_1 \gamma_n + \displaystyle\int_0^\infty g(s) e^{-i\lambda s} ds\gamma_n^\alpha & k\gamma_n \\
				k\gamma_n^\beta & \lambda^2 - \alpha_2 \gamma_n
			\end{pmatrix}
			\begin{pmatrix}
				k_1 \\
				k_2
			\end{pmatrix}
			=
			\begin{pmatrix}
				0 \\
				1
			\end{pmatrix}.
		\end{align}
		Let $P_1(s)\triangleq s-\alpha_1\gamma_n$, $P_2(s)\triangleq s-\alpha_2\gamma_n$, and
		\begin{equation}\label{ilamda}
			I_\lambda\triangleq\int_0^\infty g(s)e^{-i\lambda s}ds
			=\int_0^\infty g(0)e^{-(\delta+i\lambda)s}ds
			=\frac{g(0)}{\delta+i\lambda}.
		\end{equation}
Let $P_1(s)P_2(s)-k^2\gamma_n^{\beta+1}=0$. We get
		\begin{align}\label{p1s}
		(s-\alpha_1\gamma_n)(s-\alpha_2\gamma_n)-k^2\gamma_n^{\beta+1} =s^2-(\alpha_1+\alpha_2)\gamma_n s+\alpha_1\alpha_2\gamma_n^2-k^2\gamma_n^{\beta+1}=0.
		\end{align}
		Solving equation \eqref{p1s} yields two solutions
		\begin{equation}\label{sn}
			s_n^\pm \triangleq \frac{(\alpha_1+\alpha_2)\gamma_n \pm \sqrt{(\alpha_1-\alpha_2)^2\gamma_n^2 + 4k^2\gamma_n^{\beta+1}}}{2}.
		\end{equation}
It is obvious $s_n^+>0$ and $s_n^->0$ for $n\gg1$ since $\beta<1$.

First we consider the case $\alpha_1<\alpha_2. $
Let $\lambda\triangleq\lambda_n=\sqrt{s_n^+}$, we get
\begin{align*}
	&\lim_{n\to\infty}\lambda_n|I_{\lambda_n}|=\lim_{n\to\infty}\frac{g(0)}{\sqrt{\frac{\delta^2}{\lambda_n^2}+1}}=g(0),\\
	&\lim_{n\to\infty}\frac{\lambda_n^2}{\gamma_n}=\lim_{n\to\infty}\frac{s_n^+}{\gamma_n}=\frac12\left(\alpha_1+\alpha_2+|\alpha_1-\alpha_2|\right)=\alpha_2,\\
	&\lim_{n\to\infty}\frac{p_1(\lambda_n^2)}{\gamma_n}=\lim_{n\to\infty}\frac{s_n^+}{\gamma_n}-\alpha_1=\alpha_2-\alpha_1,\\
	&\lim_{n\to\infty}\frac{p_2(\lambda_n^2)}{\gamma_n^\beta}=\lim_{n\to\infty}\frac{k^2\gamma_n}{p_1(\lambda_n^2)}=\frac{k^2}{\alpha_2-\alpha_1}.
\end{align*}
Since $P_1(s_n^+)P_2(s_n^+) - k^2 \gamma_n^{\beta+1}=0$,  from \eqref{lamdak}, we obtain
\begin{equation}\label{k2n}
	k_2\triangleq k_{2,n} = \frac{P_1(\lambda_n^2) + I_{\lambda_n} \gamma_n^\alpha}{ I_{\lambda_n} \gamma_n^\alpha P_2(\lambda_n^2) }.
\end{equation}
Since
\begin{align*}
	&\frac{\lambda_n^{2\alpha+2\beta-3}P_1(\lambda_n^2)}{ |I_{\lambda_n}| \gamma_n^\alpha P_2(\lambda_n^2) }-\frac{\lambda_n^{2\alpha+2\beta-3}}{P_2(\lambda_n^2) }\le\lambda_n^{2\alpha+2\beta-3}|k_{2,n}|\le\frac{\lambda_n^{2\alpha+2\beta-3}P_1(\lambda_n^2)}{ |I_{\lambda_n}| \gamma_n^\alpha P_2(\lambda_n^2) }+\frac{\lambda_n^{2\alpha+2\beta-3}}{P_2(\lambda_n^2) },\\
	&\lim_{n\to\infty}\frac{\lambda_n^{2\alpha+2\beta-3}P_1(\lambda_n^2)}{ |I_{\lambda_n}| \gamma_n^\alpha P_2(\lambda_n^2) }=\lim_{n\to\infty}\left(\frac{\lambda_n^2}{\gamma_n}\right)^{\alpha+\beta-1}\frac{P_1(\lambda_n^2)/\gamma_n}{\lambda_n|I_{\lambda_n}|P_2(\lambda_n^2)/\gamma_n^\beta}=\alpha_2^{\alpha+\beta-1}\frac{(\alpha_2-\alpha_1)^2}{g(0)k^2},\\
	&\lim_{n\to\infty}\frac{\lambda_n^{2\alpha+2\beta-3}}{P_2(\lambda_n^2)}=\lim_{n\to\infty}\frac{1}{P_2(\lambda_n^2)/\gamma_n^\beta}\left(\frac{\lambda_n^2}{\gamma_n}\right)^\beta\lambda_n^{2\alpha-3}=0,
\end{align*}
Then it follows
\begin{align*}
	\lim_{n\to\infty}\lambda_n^{2\alpha+2\beta-3}|k_{2,n}|=\alpha_2^{\alpha+\beta-1}\frac{(\alpha_2-\alpha_1)^2}{g(0)k^2}.
\end{align*}

By the definition of norm $\|\cdot\|_{\mathcal{H}}$ and \eqref{vp}, we have
\[
\|U_n\|_{\mathcal{H}}\ge\|q_n\|=\|i\lambda_n p_n\|=\lambda_n|k_{2,n}|\|e_n\|=\lambda_n|k_{2,n}|.
\]
Then we get
\begin{equation*}
	\liminf_{n\to\infty}\lambda_n^{2\alpha+2\beta-4}\|U_n\|_{\mathcal{H}}\ge \lim_{n\to\infty}\lambda_n^{2\alpha+2\beta-3}|k_{2,n}|=\alpha_2^{\alpha+\beta-1}\frac{(\alpha_2-\alpha_1)^2}{g(0)k^2}.
\end{equation*}
The above estimate and Theorem \ref{thmBT} imply if the polynomial decay rate is $t^{-\delta}$,  then $\delta\le\dfrac{1}{4-2\beta-2\alpha}.$

Second we consider the case $\alpha_1>\alpha_2. $
Let $\lambda\triangleq\lambda_n=\sqrt{s_n^-}$, we get
\begin{align*}
	&\lim_{n\to\infty}\lambda_n|I_{\lambda_n}|=\lim_{n\to\infty}\frac{g(0)}{\sqrt{\frac{\delta^2}{\lambda_n^2}+1}}=g(0),\\
	&\lim_{n\to\infty}\frac{\lambda_n^2}{\gamma_n}=\lim_{n\to\infty}\frac{s_n^-}{\gamma_n}=\frac12\left(\alpha_1+\alpha_2-|\alpha_1-\alpha_2|\right)=\alpha_2,\\
	&\lim_{n\to\infty}\frac{p_1(\lambda_n^2)}{\gamma_n}=\lim_{n\to\infty}\frac{s_n^-}{\gamma_n}-\alpha_1=\alpha_2-\alpha_1,\\
	&\lim_{n\to\infty}\frac{p_2(\lambda_n^2)}{\gamma_n^\beta}=\lim_{n\to\infty}\frac{k^2\gamma_n}{p_1(\lambda_n^2)}=\frac{k^2}{\alpha_2-\alpha_1}.
\end{align*}
According to the similar calculation as above, if the polynomial decay rate is $t^{-\delta}$, it also requires $\delta\le\dfrac{1}{4-2\beta-2\alpha}.$\end{proof}

\begin{proof}[Proof of \eqref{remarkeq1} of Remark \ref{rem-sharpness}] Since $\alpha_1=\alpha_2 $, let $\lambda\triangleq\lambda_n=\sqrt{s_n^+}$, we get
\begin{align*}
 s_n^+=\alpha_1\gamma_n+k\gamma_n^{(\beta+1)/2}=\alpha_2\gamma_n+k\gamma_n^{(\beta+1)/2},
\end{align*}
  which imply
\begin{align*}
	&\lim_{n\to\infty}\lambda_n|I_{\lambda_n}|=\lim_{n\to\infty}\frac{g(0)}{\sqrt{\frac{\delta^2}{\lambda_n^2}+1}}=g(0),\\
	&\lim_{n\to\infty}\frac{p_1(\lambda_n^2)}{\gamma_n^{(\beta+1)/2}}=\lim_{n\to\infty}\frac{p_2(\lambda_n^2)}{\gamma_n^{(\beta+1)/2}}=k.
\end{align*}
Since $P_1(s_n^+)P_2(s_n^+) - k^2 \gamma_n^{\beta+1}=0$,  from \eqref{lamdak}, we obtain
\begin{equation}\label{k1n}
	k_1 \triangleq k_{1,n} = \frac{-k \gamma_n}{ I_{\lambda_n} \gamma_n^\alpha P_2(\lambda_n^2) }.
\end{equation}
Then it follows
\begin{align*}
	\lim_{n\to\infty} \lambda_n^{2\alpha+\beta-2} |k_{1,n}|
	&= \lim_{n\to\infty} \lambda_n^{2\alpha+\beta-2} \frac{k\gamma_n}{\left|I_{\lambda_n}\right| \gamma_n^\alpha P_2(\lambda_n^2)} \\
	&= \lim_{n\to\infty} \frac{k}{\lambda_n \left|I_{\lambda_n}\right| P_2(\lambda_n^2) /\gamma_n^{\frac{\beta+1}{2}}} \\
	&= \frac{k}{g(0)k}
	= \frac{1}{g(0)}.
\end{align*}
By the definition of norm $\|\cdot\|_{\mathcal{H}}$ and \eqref{vp}, we have
\[
\|U_n\|_{\mathcal{H}}\ge\|u_n\|=\|i\lambda_n v_n\|=\lambda_n|k_{1,n}|\|e_n\|=\lambda_n|k_{1,n}|.
\]
Then we get
\begin{equation*}
	\liminf_{n\to\infty}\lambda_n^{2\alpha+\beta-3}\|U_n\|_{\mathcal{H}}\ge \lim_{n\to\infty}\lambda_n^{2\alpha+\beta-2}|k_{1,n}|=\frac{1}{g(0)}.
\end{equation*}
The above estimate and Theorem \ref{thmBT} imply if the polynomial decay rate is $t^{-\delta}$,  then $\delta\le\dfrac{1}{3-\beta-2\alpha}.$
	\end{proof}


\end{document}